\documentclass[11pt]{article}  
\usepackage{amsmath}
\usepackage{amssymb}
\usepackage{theorem}
\usepackage{euscript}

\newtheorem{theorem}{Theorem}[section]
\newtheorem{lemma}{Lemma}[section]
\newtheorem{corollary}{Corollary}[section]
\numberwithin{equation}{section}
\def\II{{\mathbb I}}

\def\ZZ{{\mathbb Z}}
\def\NN{{\mathbb N}}
\def\RR{{\mathbb R}}
\def\GG{{\mathbb G}}
\def\TT{{\mathbb T}}
\def\IId{{\mathbb I}^d}

\def\RRd{{\mathbb R}^d}
\def\TTd{{\mathbb T}^d}
\def\GGd{{\mathbb G}^d}
\def\RRdp{{\mathbb R}^d_+}
\def\ZZdp{{\mathbb Z}^d_+}

\def\ZZdpe{{\mathbb Z}^d_+(e)}
\def\RRbp{{\bar{\mathbb R}}^3_+}
\def\Pp{{\mathcal P}}

\def\Hab{H^{\alpha, \beta}}
\def\Hg{H^\gamma}

\def\BO{B^{\Omega}_{p,\theta}}
\def\BOe{B^{\Omega}_{p,\theta}(e)}

\def\Ba{B^a_{p,\theta}}
\def\Bg{B^\alpha_{p,\theta}}

\def\Bgq{B^\gamma_{q,\tau}}
\def\Bab{B^{\alpha,\beta}_{p,\theta}}
\def\Ua{U^a_{p,\theta}}
\def\Uab{U^{\alpha,\beta}_{p,\theta}}
\def\Bpsi{B^{\{\psi\}}_{p,\theta}}

\title{\sffamily Sampling and cubature on sparse grids based on a B-spline quasi-interpolation 
\author{ 
Dinh D\~ung \\[5mm]
Vietnam National University, Hanoi, Information Technology Institute \\
144 Xuan Thuy, Cau Giay, Hanoi, Vietnam\\
{\ttfamily dinhzung@gmail.com}\\[4mm]}}
\date{\ttfamily  July 09, 2015 -  Version R13}
\begin{document}
\maketitle
\begin{abstract}
Let $X_n = \{x^j\}_{j=1}^n$ be a set of $n$ points in the $d$-cube $\IId:=[0,1]^d$, and 
$\Phi_n = \{\varphi_j\}_{j =1}^n$  a family of  $n$ functions on $\IId$. 
 We consider the approximate recovery of functions $f$ on ${\II}^d$ from the sampled values 
$f(x^1), ..., f(x^n)$, by the linear sampling algorithm 
$
L_n(X_n,\Phi_n,f) 
:= 
\sum_{j=1}^n f(x^j)\varphi_j.
$
The error of sampling recovery is measured in the norm of the space $L_q(\IId)$-norm or the energy quasi-norm of the isotropic Sobolev space $W^\gamma_q(\IId)$ for $1 < q < \infty$ and $\gamma > 0$. Functions $f$ to be recovered are from the unit ball in  Besov type spaces of an anisotropic smoothness, in particular, spaces $B^{\alpha,\beta}_{p,\theta}$ of a ``hybrid" of mixed smoothness $\alpha > 0$ and  isotropic smoothness 
$\beta \in \RR$, and  spaces $\Ba$ of a nonuniform  mixed smoothness $a \in \RRdp$. We constructed asymptotically optimal linear sampling algorithms $L_n(X_n^*,\Phi_n^*,\cdot)$ on special sparse grids $X_n^*$ and a family $\Phi_n^*$ of linear combinations of integer or half integer translated dilations of tensor products of B-splines. We computed the asymptotic order of the error of the optimal recovery. This construction is based on B-spline quasi-interpolation representations of functions in $B^{\alpha,\beta}_{p,\theta}$ and $\Ba$. As consequences we obtained the asymptotic order of optimal cubature formulas for numerical integration of functions from the unit ball of these Besov type spaces.


 
\medskip
\noindent
{\bf Communicated by Albert Cohen}  
\end{abstract}

\section{Introduction} 
The aim of the present paper is to construct  linear sampling algorithms and cubature formulas on sparse grids based on a B-spline quasi-interpolation, and study their optimality in the sense of asymptotic order for functions on the unit $d$-cube ${\II}^d:= [0,1]^d$, having an anisotropic smoothness. The error of sampling recovery is measured in the norm of the space $L_q(\IId)$-norm or the  energy norm of the isotropic Sobolev space $W^\gamma_q(\IId)$ for $1 < q < \infty$ and $\gamma > 0$. For convenience, we use somewhere the convention $W^0_q(\IId):= L_q(\IId)$.  

Let $X_n = \{x^j\}_{j=1}^n$ be a set of $n$ points 
in ${\II}^d$, $\Phi_n = \{\varphi_j\}_{j =1}^n$ a family of  
$n$ functions on ${\II}^d$. If  $f$ is a function on ${\II}^d$, for approximately recovering $f$ from the sampled values $f(x^1),..., f(x^n)$, we  define the linear sampling algorithm $L_n(X_n,\Phi_n,\cdot)$  by 
\begin{equation} \label{def[L_n]}
L_n(X_n,\Phi_n,f) 
:= \ \sum_{j=1}^n f(x^j)\varphi_j.
\end{equation}
Let $B$ be a quasi-normed space of functions on $\IId$, equipped with the quasi-norm 
$\|\cdot\|_B$. For $f \in B$,  we measure the recovery error by 
$\|f - L_n(X_n,\Phi_n,f)\|_B$. Let $W \subset B$. To study  optimality of linear sampling algorithms 
of the form \eqref{def[L_n]} for 
recovering $f \in W$ from $n$ of their values, we will use the quantity of optimal sampling recovery
\begin{equation} \nonumber 
r_n(W,B) 
\ := \ \inf_{X_n, \Phi_n} \  \sup_{f \in W} \, \|f - L_n(X_n,\Phi_n,f)\|_B. 
\end{equation}

Further, let $\Lambda_n = \{\lambda_j\}_{j =1}^n$ be a sequence of $n$ numbers. For  a $f \in C(\IId)$, we  want to approximately compute the integral 
\begin{equation*}
I(f)
:= \
\int_{\IId} f(x) \ dx 
\end{equation*}
by the cubature formula
\begin{equation*}
I_n(X_n,\Lambda_n,f)
:= \
\sum_{j=1}^n \lambda_j f(x^j).
\end{equation*}
To study the optimality of cubature formulas for 
$f \in W$, we use the quantity of optimal cubature
\begin{equation*} 
i_n(W) 
\ := \ \inf_{X_n, \Lambda_n} \  \sup_{f \in W} \, |I(f) - I_n(X_n,\Lambda_n,f)|. 
\end{equation*}

Recently, there has been increasing interest in solving approximation and numerical problems that involve
functions depending on a large number $d$ of variables.  Without further assumptions the computation time typically grows exponentially in $d$, and
the problems become intractable already for mild dimensions $d$. 
This is the so called curse of dimensionality \cite{Be57}. 
In sampling recovery and numerical integration, a classical model in attempt to overcome it which has been widely studied,  is to impose certain mixed smoothness or more general anisotropic smoothness conditions on the function to be approximated, and to employ sparse grids for construction of approximation algorithms for sampling recovery or integration. 
We refer the reader to  \cite{BG04, GeGr08, NW08, NW10} for surveys and the references therein on various aspects of this direction. 
 
Sparse grids for sampling recovery and numerical integration were first considered by Smolyak \cite{Sm63}. He constructed the following grid of dyadic points 
\begin{equation*} 
\Gamma(m)
:= \ 
\{ 2^{-k}s: k \in D(m),\ s \in I^d(k)\},
\end{equation*}
where
$D(m) := \{k \in {\ZZ}^d_+: |k|_1 \le m\}$ and  
$I^d(k):= \{s \in {\ZZ}^d_+: 0 \le s_i \le 2^{k_i}, \ i \in [d]\}$. 
Here and in what follows, we use the notations:
$xy := (x_1y_1,..., x_dy_d)$; 
$2^x := (2^{x_1},...,2^{x_d})$;
$|x|_1 := \sum_{i=1}^d |x_i|$ for $x, y \in {\RR}^d$;
$[d]$ denotes the set of all natural numbers from $1$ to $d$; $x_i$ denotes the $i$th coordinate 
of $x \in \RR^d$, i.e., $x := (x_1,..., x_d)$. Observe that $\Gamma(m)$ is a sparse grid
of the size $2^m m^{d-1}$ in comparing with the standard full grid of the size $2^{dm}$.

In approximation theory, Temlyakov \cite{Te85} -- \cite{Te93} and the author of the present paper 
\cite{Di90} -- \cite{Di92} developed Smolyak's construction for studying the asymptotic order of $r_n(W, L_q(\TTd))$ for periodic Sobolev classes $W^a_p$ and Nikol'skii classes $H^a_p$ having nonuniform mixed smoothness $a=(a_1,...,a_d) \in \RRd$ with different $a_j > 0$, where 
$\TTd$ denotes the $d$-dimensional torus. For the uniform mixed smoothness $\alpha{\bf 1}$, 
Temlyakov~\cite{Te93b} investigated sampling recovery for periodic Sobolev classes $W^{\alpha{\bf 1}}_p$ and Nikol'skii classes $H^{\alpha{\bf 1}}_p$, and recently, Sickel and Ullrich \cite{SU07} for periodic Besov classes 
$U^{\alpha{\bf 1}}_{p,\theta}$ , where ${\bf 1} := (1,1,...,1) \in \RRd$.
For non-periodic functions of mixed smoothness linear sampling algorithms have been recently studied by Triebel \cite{Tr10} $(d=2)$, D\~ung \cite{Di11}, Sickel and Ullrich \cite{SU11}, using the mixed tensor product of B-splines and Smolyak grids $\Gamma(m)$. 
Smolyak grids are a counterpart of  hyperbolic crosses which are frequency domains of trigonometric polynomials widely used for approximations of functions with a bounded mixed smoothness. These  hyperbolic cross trigonometric approximations are initiated by Babenko \cite{Ba60a}. For further surveys and references on the topic see \cite{Di86, Te86, Te93b, DU13}, and the more recent contributions \cite{SU07, U08}. 

In computational mathematics, the sparse grid approach was first considered by Zenger \cite{Ze91}. Numerical integration using sparse grids was investigated in \cite{GeGr98}. 
For non-periodic functions of mixed smoothness of integer order, linear sampling algorithms on sparse grids have been investigated by Bungartz and Griebel \cite{BG04} employing hierarchical Lagrangian polynomials multilevel basis  and measuring the approximation error in the $L_2$-norm and energy $H^1$-norm. 
There is a very large number of papers on sparse grids in various problems of approximations, sampling recovery and integration with applications in data mining, mathematical finance, learning theory, numerical solving of PDE and stochastic PDE, etc. to mention all of them.  The reader can see the surveys in 
\cite{BG04, GN09, GeGr08} and the references therein.  For recent further developments and results see in 
\cite{GHo10, GH13a, GH13b, GaHe09, BGGK13}.

Quasi-interpolation based on scaled B-splines with integer knots, possesses good local and approximation properties for smooth functions, see \cite[p. 63--65]{BHR}, \cite[p. 100-107]{C92}. It can be an efficient tool in some high-dimensional approximation problems, especially in applications ones. Thus, one of the important bases for sparse grid high-dimensional approximations having various applications, are the Faber functions (hat functions) which are piecewise linear B-splines of second order  
\cite{BG04,GeGr08, GHo10, GH13a, GH13b, GaHe09, BGGK13}.  The representation by Faber basis can be obtained by the B-spline quasi-interpolation (see, e. g., \cite{Di11}).                                              
In the recent paper \cite{Di11}, by using a quasi-interpolation representation of functions by mixed high-order
B-spline series, we constructed linear sampling algorithms $L_n(X_n,\Phi_n,f)$ on Smolyak grids $\Gamma(m)$, 
for functions on ${\II}^d$ from the non-periodic Besov class
$U^{\alpha{\bf 1}}_{p,\theta}$, which is defined as the unit ball of the Besov space 
$B^{\alpha{\bf 1}}_{p,\theta}$  of functions on $\IId$ having uniform mixed smoothness $\alpha$. For various 
$0<p,\theta, q \le \infty$ and $\alpha > 1/p$,
we proved upper bounds for the worst case error $\sup_{f \in U^{\alpha{\bf 1}}_{p,\theta}}\|f - L_n(X_n,\Phi_n,f)\|_q$ which  in some cases, coincide with 
the asymptotic order 
\begin{equation} \label{r_n[beta=0]}
r_n(U^{\alpha{\bf 1}}_{p,\theta}, L_q(\IId))
\ \asymp \
n^{-\alpha + (1/p - 1/q)_+} \log_2^{(d-1)b} n,
\end{equation}
where $b = b(\alpha, p, \theta, q) > 0$ and $x_+ := \max (0,x)$ for $x \in \RR$.

In the paper \cite{DU13}, we obtained the asymptotic order of optimal sampling recovery on Smolyak grids in the $L_q(\IId)$-quasi-norm of functions from $U^{\alpha{\bf 1}}_{p,\theta}$ for 
$0<p,\theta, q \le \infty$ and $\alpha > 1/p$. It is necessary to emphasize that any sampling algorithm on Smolyak grids always gives a lower bound of recovery error of the form as in the right side of \eqref{r_n[beta=0]} with the logarithm term $\log_2^{(d-1)b} n$, $b>0$. Unfortunately, in the case when the dimension $d$ is very large and the number $n$ of samples is rather mild, the main term becomes $\log_2^{(d-1)b} n$ which grows fast exponentially in $d$. To avoid this exponential grow we impose to functions other anisotropic smoothnesses and construct appropriate sparse grids for functions having them. Namely, we  extend the above study to functions on ${\II}^d$ from the  classes
$\Uab$ for $\alpha > 0, \beta \in \RR$, and  $\Ua$ for $a \in \RRd$ with $a_1 < a_2 \le \cdots \le a_d$, which are defined as the unit ball of the Besov type  spaces $\Bab$ and   $\Ba$, respectively. The space $\Bab$ and 
 $\Ba$ are certain sets of functions with bounded mixed modulus of smoothness. Both of them are  generalizations in different ways of the space $B^{\alpha {\bf 1}}_{p,\theta}$ of mixed smoothness $\alpha$. The space $\Bab$ is a ``hybrid" of the space $B^{\alpha {\bf 1}}_{p,\theta}$ and the classical isotropic Besov space $B^\beta_{p,\theta}$ of smoothness $\beta$. 

The space $\Bab$ is a Besov type generalization of the Sobolev type space 
$\Hab= B^{\alpha,\beta}_{2,2}$.  The latter space has been introduced in \cite{GN09} for solutions of the following elliptic variational problems 
$a(u,v)\ = \ (f,v) \ {\rm for \ all} \ v \in \Hg,$
where $f \in H^{-\gamma}$ and $a: \Hg \times \Hg\to \RR$ is a bilinear symmetric form satisfying the conditions 
$ a(u,v) \le \lambda \|u\|_{\Hg}\|v\|_{\Hg}$ and $a(u,u) \ge \mu \|u\|_{\Hg}^2$.
 By use of tensor-product biorthogonal wavelet bases, the authors of these papers constructed so-called optimized sparse grid  subspaces for finite element approximations of the solution having $\Hab$-regularity, whereas the approximation error is measured in the energy norm of isotropic Sobolev space $\Hg$. They generalized the construction of \cite{BG99} for a hyperbolic cross approximation of the solution of Poisson's equation to  elliptic variational problems.  
A generalization $\Hab\big((\RR^3)^N\big)$ of the space $\Hab$ of functions on $(\RR^3)^N$, based on isotropic Sobolev smoothness of the space $H^1(\RR^3)$, has been considered by Yserentant \cite{Ys07}--\cite{Ys11} for solutions 
$u:(\RR^3)^N \to \RR: ({\bf x}_1,...,{\bf x}_N) \mapsto u({\bf x}_1,...,{\bf x}_N)$ of the electronic Schr\"odinger equation $Hu = \lambda u$ for eigenvalue problem where $H$ is the Hamilton operator. 
He proved that the eigenfunctions are contained in the intersection of spaces
\[
H^{1,0}\big((\RR^3)^N\big) \cap \big[\cap_{\vartheta < 3/4} H^{\vartheta,1}\big((\RR^3)^N\big)\big].
\]
In numerical solving by hyperbolic cross approximations the error is measured in the norm of the space $L_2\big((\RR^3)^N\big)$ and the energy norm of the isotropic Sobolev space $H^1\big((\RR^3)^N\big)$.  
See also \cite{GH10a}--\cite{GH13b}, \cite{KY12} for further results and developments. 

All the above remarks and comments tell us about a motivation to construct efficient
linear sampling algorithms and cubature formulas on sparse grids based on a high-order B-spline quasi-interpolation, for functions having anisotropic smoothness from $\Bab$ and  $\Ba$, measuring the approximation error in the quasi-norm of $L_q(\IId)$ or the energy quasi-norm of $W^\gamma_q(\IId)$.
The optimality of these algorithms and formulas will be studied in terms of the quantities 
$r_n(\Uab, W^\gamma_q(\IId))$ and $i_n(\Uab)$ for the case $\beta \not= \gamma$, and $r_n(\Ua, L_q(\IId))$ and $i_n(\Ua)$ for the case of nonuniform mixed smoothness $a$ with 
$a_1 < a_2 \le \cdots \le a_d$. 

In the following, as an example, let us mention one of our main results. 
For a set $\Delta \subset \ZZdp$, we define the grid points in $\IId$
$G(\Delta) := \{ 2^{-k}s: k \in \Delta,\ s \in I^d(k)\}$,
and the linear sampling algorithms of the form 
\begin{equation} \label{eq[L_n(f)]}
L_n^\Delta(X_n,\Phi_n,f) 
\ = \ 
\sum_{k \in \Delta} \ \sum_{j \in I^d(k)} f(2^{-k}j) \psi_{k,j}, 
\end{equation} 
where $n:= |G(\Delta)|$, $X_n := G(\Delta)$, 
$\Phi_n:= \{ \psi_{k,j}\}_{k \in \Delta, \, j \in I^d(k)}$ and 
 $\psi_{k,j}$ are explicitly constructed as linear combinations of  
at most $N$ of B-splines $M^{(r)}_{k,s}$ for some $N$ independent of $k,j$ and $f$, $M^{(r)}_{k,s}$ are tensor products of either integer or half integer translated dilations of the centered B-spline of order $r$.

Let $ 0 < p, \theta, q \le \infty$, $\alpha, \gamma \in \RR_+$, $\beta \in \RR$ satisfying the conditions
$\min(\alpha,\alpha + \beta) > 1/p$ and $\alpha  > (\gamma -\beta)/d$ if $\beta > \gamma$, and 
$\alpha > \gamma - \beta$ if $ \beta < \gamma$ (with the additional restriction 
$1 < q < \infty$ in the case $\gamma > 0$).  Then  we explicitly constructed a set $\Delta_n$ such that $|G(\Delta_n)| \le n$ and
\begin{equation} \label{r_n[Uab,b<0]}
\sup_{f \in \Uab} \|f - L_n^{\Delta_n}(X_n,\Phi_n,f)\|_{W^\gamma_q(\IId)}
 \asymp  
r_n(\Uab, W^\gamma_q(\IId)) 
 \asymp  
\begin{cases}
n^{- \alpha  - (\beta-\gamma)/d + (1/p - 1/q)_+},  &  \beta > \gamma, \\
n^{- \alpha  - \beta + \gamma +(1/p - 1/q)_+},  &  \beta < \gamma;
\end{cases}
\end{equation}
From \eqref{r_n[Uab,b<0]} for the case $\gamma = 0$, $p=1$ we derived that
\begin{equation} \nonumber
i_n(\Uab) 
\ \asymp \
\begin{cases}
n^{- \alpha  - \beta/d + (1/p - 1)_+}, \ &  \beta > 0, \\
n^{- \alpha  - \beta +(1/p - 1)_+}, \ &  \beta < 0.
\end{cases}
\end{equation}
The set $\Delta_n$ is specially constructed for the class of $\Uab$, depending on the relationship between  $0<p,\theta, q, \tau \le \infty$ and $\alpha, \beta$ respectively. 
{\em The grids $G(\Delta_n)$ are sparse and have much smaller number of sample points
than the corresponding standard full grids and Smolyak grids, but give the same error of the sampling recovery on the both latter ones.} The construction of asymptotically optimal linear sampling algorithms 
$L_n^{\Delta_n}(X_n,\Phi_n,\cdot)$ is essentially  
based on quasi-interpolation representations by B-spline series of functions $f \in \Bab$ with a discrete equivalent quasi-norm in terms of the coefficient function-valued functionals of this series. Moreover, for the sampling recovery in the $L_1$-norm,  $L_n^{\Delta_n}(X_n,\Phi_n,\cdot)$ generates an asymptotically optimal cubature formula (see Section~\ref{Numerical integration} for details). 

To discuss results on the class $\Ua$ we preliminarily notice the following. For the nonuniform mixed smoothness $a$ with $0 < a_1 = \cdots = a_\nu < a_{\nu + 1} \le \cdots \le a_d$, it is known that in many approximation problems asymptotic characteristics of corresponding function classes with smoothness $a$ the extra $\log n$ appears in the form $(\log n)^{(\nu-1)b}$ (see, for example, \cite{Te86, Di86} and references there). In the case $\nu = 1$, the extra $\log n$ disappears independently of $b$. This makes the problem of finding the optimal rate in the case $\nu = 1$ much easier than in the case $\nu > 1$. Thus, it was proven in 
\cite{Te85} that for $1 \le p \le \infty$, $r > 1/p$,
\begin{equation} \label{r_n[Ua]<(periodic)}
r_n(U^a_{p, \infty}(\TT^d), L_q(\TTd)) 
\ \ll \
n^{- a_1}(\log n)^{(\nu-1)(a_1 + 1)}.
\end{equation}
Combining this with the well-known embedding 
$B^a_{p, \theta} \hookrightarrow  B^{a-(1/p-1/q)_+}_{q,\infty}$
and well-known lower bounds in the univariate case, we obtain for the case $\nu = 1$,
\begin{equation} \nonumber
r_n(\Ua(\TT^d), L_q(\TTd)) 
\ \asymp \
n^{- a_1 + (1/p - 1/q)_+} \ (0 < \theta \le \infty).
\end{equation}
It is important to emphasize that linear sampling algorithms  constructed in \cite{Te85} which give the upper bounds \eqref{r_n[Ua]<(periodic)} and which are asymptotically optimal for the case $\nu = 1$, are developed from a construction in \cite{Sm63}, but essentially based on extended nonuniform Smolyak grids.  These grids are a counterpart of extended hyperbolic crosses suggested by Teljakovskii \cite{Tel64} (for further development of Teljakovskii's construction in hyperbolic cross approximation of functions having one or several nonuniform mixed smoothnesses, see \cite{Te86, Di86} for surveys and references there).  Extended nonuniform Smolyak grids and their modifications then were used in sampling recovery problems in \cite{Te85} -- \cite{Te93}, \cite{Di90} -- \cite{Di92}, \cite{GH13a} -- \cite{GH13b}.

In the present paper, we are interested in constructing asymptotically optimal linear sampling algorithms for the nonperiodic Besov class $\Ua$ with nonuniform mixed smoothness $a$.  More precisely, if $ 0 < p, \theta, q \le \infty$ and $a \in \RRd$ with $1/p < a_1 < a_2 \le ... \le a_d$, we explicitly constructed a set $\Delta'_n$ such that $|G(\Delta'_n)| \le n$ and the sampling algorithm $L_n^{\Delta'_n}(X_n,\Phi_n,f)$ is  asymptotically  optimal for the class $\Ua$, i. e.,
\begin{equation} \label{r_n[Ua]><}
\sup_{f \in \Ua} \|f - L_n^{\Delta'_n}(X_n,\Phi_n,f)\|_q
 \asymp 
r_n(\Ua, L_q(\IId)) 
\ \asymp \
n^{- a_1 + (1/p - 1/q)_+}.
\end{equation}
 The construction of the sampling algorithms $L_n^{\Delta'_n}(X_n,\Phi_n,f)$ on the grids $G(\Delta'_n)$ is similar to that in \cite{Te85} for the case $\nu = 1$. The main contribution of the present  paper is  a theorem on quasi-interpolation representation by B-spline series of functions $f \in \Ba$ with a discrete equivalent quasi-norm in terms of the coefficient function-valued functionals of this series. This theorem plays a key role in constructing the asymptotically optimal linear sampling algorithms $L_n^{\Delta'_n}(X_n,\Phi_n,f)$ for the class $\Ua$, as well in proving the relation \eqref{r_n[Ua]><}.

In the present paper, we consider only two kinds of anisotropic smoothness spaces $\Bab$ and $\Ba$. However, our constructions and methods of proofs of results can be extended to other kinds of anisotropic smoothness, see examples in Remark at the end of Section \ref{Quasi-interpolant}. We are restricted to compute the asymptotic order of $r_n$ 
with respect only to $n$ when 
$n \to \infty$, not analyzing  the dependence on the number of variables $d$. Recently, in \cite{DU12} Kolmogorov $n$-widths $d_n(U,\Hg)$ and $\varepsilon$-dimensions $n_\varepsilon(U,\Hg)$ in space $\Hg$ of periodic multivariate function classes $U$ have been investigated in high-dimensional settings, where $U$ is the unit ball in $\Hab$ or its subsets. We computed the accurate dependence of $d_n(U,\Hg)$ and  $n_\varepsilon(U,\Hg)$ as a function of two variables $n$, $d$ or $\varepsilon$, $d$.  Although $n$ is the main parameter in the study of convergence rate with respect to $n$ when $n \to \infty$, the parameter $d$ may affect this rate when $d$ is large. It is interesting and important to investigate  optimal sampling recovery  and cubature in such high-dimensional settings. In the recent paper \cite{Di15}, we have constructed linear algorithms of sampling recovery and cubature formulas on Smolyak grids of periodic $d$-variate functions having Lipschitz-H\"older mixed smoothness based on B-spline quasi-interpolation, and  established upper and lower estimates of the error of the optimal sampling recovery and the optimal integration on Smolyak grids, explicit in $d$ and $n$ when the number $d$ of variables and the number $n$ of sampled function values may be very large. 

The present paper is organized as follows. 
In Section \ref{Quasi-interpolant}, we give  definitions of Besov type spaces $\BO$ of functions with bounded mixed 
modulus of smoothness, in particular, spaces $\Bab$ and  $\Ba$,  and prove  theorems on 
quasi-interpolation representation by B-spline series,
with relevant discrete equivalent quasi-norms.
In Section \ref{Sampling recovery}, we construct linear sampling algorithms on sparse grids of the form 
\eqref{eq[L_n(f)]} for function classes $\Uab$ and $\Ua$, and prove upper bounds for the error of recovery by these algorithms. In Section \ref{Optimality}, we prove the sparsity and asymptotic optimality of the linear sampling algorithms constructed in Section \ref{Sampling recovery}, for the quantities  
$r_n(\Uab, L_q(\IId))$ and $r_n(\Ua, L_q(\IId))$, and establish their asymptotic orders. In Section \ref{Sampling recovery in Bg}, we extend the investigations of Sections \ref{Sampling recovery} and \ref{Optimality} to the quantities $r_n(\Uab, W^\gamma_q(\IId))$ for $\gamma > 0$. 
 In Section \ref{Numerical integration}, we discuss the problem of optimal cubature formulas for numerical integration  in terms of $i_n(\Uab)$ and $i_n(\Ua)$.

\section{Function spaces and quasi-interpolation representations} 
\label{Quasi-interpolant}

\subsection{Function spaces} 
Let us first introduce 
spaces $\BO$ of functions with bounded mixed modulus of smoothness and Besov type spaces $\Bab$ and   $\Ba$ of functions with anisotropic smoothness, as well fractional isotropic Sobolev and Besov spaces
 $W^\alpha_p$ and  $B^\alpha_{p,\theta}$, and give necessary knowledge of them.
 
Let $\GG$ be a domain in $\RR$. 
For univariate functions $f$ on $\GG$ the $r$th difference operator $\Delta_h^r$ is defined by 
\begin{equation*}
\Delta_h^r(f,x) := \
\sum_{j =0}^r (-1)^{r - j} \binom{r}{j} f(x + jh).
\end{equation*}
If $e$ is any subset of $[d]$, for multivariate functions on $\GGd$
the mixed $(r,e)$th difference operator $\Delta_h^{r,e}$ is defined by 
\begin{equation*}
\Delta_h^{r,e} := \
\prod_{i \in e} \Delta_{h_i}^r, \quad \Delta_h^{r,\varnothing} := \ I,
\end{equation*}
where the univariate operator
$\Delta_{h_i}^r$ is applied to the univariate function $f$ by considering $f$ as a 
function of  variable $x_i$ with the other variables held fixed, and $I(f):= f$ for functions $f$ on $\GGd$. 

Denote by  
$L_p(\GGd)$ the quasi-normed space 
of functions on $\GGd$ with the $p$th integral quasi-norm 
$\|\cdot\|_{p,\GGd}$ for $0 < p < \infty,$ and 
the sup norm $\|\cdot\|_{\infty,\GGd}$ for $p = \infty$. 

Let
\begin{equation} \nonumber
\omega_r^e(f,t)_{p,\GGd}:= \sup_{|h_i| \le t_i, i \in e}\|\Delta_h^{r,e}(f)\|_{p,\GGd(r,h,e)}, \ t \in \RRdp,
\end{equation} 
be the mixed $(r,e)$th modulus of smoothness of $f$,
 where 
\[
\GGd(r,h,e):= \{ x \in \GGd : x_i, x_i + rh_i \in \GG, \ i \in e \}
\]
(in particular,  $\omega_r^{\varnothing}(f,t)_{p,\GGd} = \|f\|_{p,\GGd}$).

For $x,x' \in \RRd$, the inequality $x' \le x$ ($x' < x$) 
means $x'_i \le x_i \ (x'_i < x_i), \ i \in [d]$. Denote: $\RR_+ := \{x \in \RR: x \ge 0\}$.
Let $\Omega: \RRdp \to \RR_+$ be a function satisfying conditions
\begin{equation} \label{conditionI}
\Omega(t) 
\ > \ 0, \ \ t > 0, \ t \in \RRdp, 
\end{equation}
\begin{equation} \label{conditionII}
\Omega(t) 
\ \le \ 
C \Omega(t') , \ \ t \le t', \ t, t' \in \RRdp, 
\end{equation} 
and for a fixed $\gamma \in \RRdp$, $\gamma \ge {\bf 1}$, there is a constant $C' = C'(\gamma)$ such that for every
$ \lambda \in \RRdp$ with $\lambda \le \gamma$,
\begin{equation} \label{conditionIII}
\Omega(\lambda \,t) 
\ \le \ 
C' \Omega(t) , \ \   t \in \RRdp. 
 \end{equation}

For $e \subset [d]$, we define the function $\Omega_e: \RRdp \to \RR_+$ by
$\Omega_e(t)
:= \
\Omega(t^e), 
$ 
where $t^e \in \RRdp$ is given by $t^e_j = t_j$ if $j \in e$, and $t^e_j = 1$ otherwise.

If $0 <  p, \theta \le \infty$, 
we introduce the quasi-semi-norm 
$|f|_{\BOe}$ for functions $f \in L_p(\GGd)$ by
\begin{equation} \nonumber
|f|_{\BOe}:= \
\left\{
\begin{array}{cl}
\displaystyle
 \ \Big(\int_{{\II}^d} \{
\omega_r^e(f,t)_{p,\GGd}/\Omega_e(t) \}^ \theta \prod_{i \in e} t_i^{-1}dt \Big)^{1/\theta}, 
& \theta < \infty, \\[2.5ex]
\displaystyle
 \sup_{t \in \IId} \ \omega_r^e(f,t)_{p,\GGd}/\Omega_e(t),  & \theta = \infty,
\end{array} \right.
\end{equation}
(in particular, $|f|_{\BO(\varnothing)} = \|f\|_{p,\GGd}$).
For $0 <  p, \theta \le \infty$, the Besov type space 
$\BO(\GGd)$ is defined as the set of  functions $f \in L_p(\GGd)$ 
for which the  quasi-norm 
\begin{equation} \nonumber 
\|f\|_{\BO(\GGd)}
:= \ 
 \sum_{e \subset [d]} |f|_{\BOe}
\end{equation}
is finite. In what follows, we assume that the function $\Omega$ satisfies  the conditions \eqref{conditionI}--\eqref{conditionIII}.

We use the notations:
$A_n(f) \ll B_n(f)$ if $A_n(f) \le CB_n(f)$ with 
$C$ an absolute constant not depending on $n$ and/or $f \in W,$ and 
$A_n(f) \asymp B_n(f)$ if $A_n(f) \ll B_n(f)$ and $B_n(f) \ll A_n(f).$ Put  $\ZZ_+:= \{s \in \ZZ: s \ge 0 \}$ and 
$\ZZdp(e):= \{s \in \ZZdp: s_i = 0 , \ i \notin e\}$  for a set $e \subset [d]$.
\begin{lemma} \label{normequivalence[B_1]}
Let $\ 0 < p, \theta \le \infty$. Then we have the following quasi-norm equivalence
\begin{equation*} 
\|f\|_{\BO(\GGd)}
\ \asymp \ B_1(f)
\ := \
\sum_{e \subset [d]}\, \biggl(\sum_{k \in \ZZdpe} 
 \left\{\omega_r^e(f,2^{-k})_{p,\GGd}/\Omega(2^{-k})\right\}^\theta \biggl)^{1/\theta} 
\end{equation*}
with the corresponding change to sup when $\theta = \infty$.
\end{lemma}

\begin{proof}
This lemma follows from properties of mixed modulus of smoothness $\omega_r^e(f,t)_{p,\GGd}$ and the properties 
\eqref{conditionI}--\eqref{conditionIII} of the function $\Omega$. 
We prove it for completeness. The lemma will be proven if we show that for every $e \subset [d]$,
\begin{equation} \label{|f|_BOe><}
|f|_{\BOe}
\ \asymp \ 
 \biggl(\sum_{k \in \ZZdpe} 
  \left\{\omega_r^e(f,2^{-k})_p/\Omega(2^{-k})\right\}^\theta \biggl)^{1/\theta}, 
\end{equation}
with the corresponding change to sup when $\theta = \infty$. Let us prove this semi-norms equivalence for instance, for $e = [d]$, $1 \le p < \infty$ and $0 < \theta < \infty$. The general case can be proven in a similar way with a slight modification. 
Put $D(k):= \{x \in \RRdp: k \le x < k + {\bf 1}\}$ and use the abbreviation
 $\omega_r(f, \cdot)_p:= \omega_r^{[d]}(f, \cdot)_p$. 
 By \eqref{conditionI}--\eqref{conditionIII} we have
\begin{equation} \label{asymp[Omega]}
\Omega(2^{-x})
\ \asymp \
\Omega(2^{-k}), \ \ x \in D(k), \ k \in \ZZdp.
\end{equation}
From the monotonicity of $\omega_r(f, \cdot)$ in each variable and the inequality
\begin{equation*} 
\omega_r(f, c\,t)_p 
\ \le \ 
\prod_{j=1}^d(1 + c_j)^r \, \omega_r(f, t)_p, \ \ c \in \RRdp, \ c > 0,
\end{equation*}
we obtain
\begin{equation} \label{asymp[omega_r]}
\omega_r(f,2^{-x})_p 
\ \asymp \ 
\omega_r(f, 2^{-k})_p,  \ \ x \in D(k), \ k \in \ZZdp.
\end{equation}
Setting $I(k) := \{t \in \IId:  2^{- k - {\bf 1}} \le t \le 2^{- k} \}$,  by 
\eqref{asymp[Omega]} and \eqref{asymp[omega_r]}  we have
\begin{equation*} 
\begin{aligned}
|f|_{\BO([d])}^\theta
\ &= \
 \ \sum_{k \in \ZZdp} \,\int_{I(k)} \{ \omega_r(f,t)_p/\Omega(t) \}^\theta \prod_{i \in [d]} t_i^{-1}dt \\ 
\ &= \
 \ \sum_{k \in \ZZdp} \,\int_{D(k)} \{ \omega_r(f,2^{-x})_p/\Omega(2^{-x}) \}^\theta dx   
\ \asymp \
 \ \sum_{k \in \ZZdp} \{ \omega_r(f, 2^{-k})_p/ \Omega(2^{-k}) \}^\theta. 
\end{aligned}
\end{equation*}
\end{proof}

Let us define the Besov type spaces $\Ba(\GGd)$ and $\Bab(\GGd)$ of functions with anisotropic smoothness as particular cases of $\BO(\GGd)$.
For $a \in \RRdp$, we define the space $\Ba(\GGd)$ of mixed smoothness $a$ by
\begin{equation} \label{def[Ba]}
\Ba(\GGd) \ := \BO(\GGd), \ \text{where $\Omega(t)= \prod_{i=1}^d t_i^{a_i}, \ t \in \RRdp$}.
\end{equation}
Let $\alpha \in \RR_+$ and $\beta \in \RR$ with $\alpha + \beta > 0$. We define the space $\Bab(\GGd)$
as follows. 
\begin{equation} \label{def[Bab]}
\Bab(\GGd) \ := \BO(\GGd), \quad \text{where} \quad
\Omega(t)= 
\left\{
\begin{array}{cl}
\displaystyle
\prod_{i=1}^d t_i^\alpha  \inf_{j \in [d]} t_j^\beta, & \ \beta \ge 0, \\[2ex]
\displaystyle
 \prod_{i=1}^d t_i^\alpha \sup_{j \in [d]} t_j^\beta, & \ \beta < 0.
\end{array} \right.
\end{equation}
The definition \eqref{def[Bab]} seems different for $\beta > 0$ and $\beta < 0$. However, it can be well interpreted in terms of the equivalent discrete quasi-norm $B_1(f)$ in Lemma \ref{normequivalence[B_1]}. Indeed, the function $\Omega$ in \eqref{def[Bab]} for both $\beta \ge 0$ and $\beta <0$ satisfies the assumptions 
\eqref{conditionI}--\eqref{conditionIII} and moreover,
\begin{equation} \nonumber
1/\Omega (2^{-x})
\ = \ 
2^{\alpha |x|_1 + \beta |x|_\infty},
\ \  x \in \RRdp, 
\end{equation}
where $|x|_\infty:= \max_{j \in [d]} |x_j|$ for $x \in \RRd$.  Hence, by Lemma \ref{normequivalence[B_1]}
 we have the following quasi-norm equivalence
\begin{equation} \label{discretenorm[Bab]}
\|f\|_{\Bab(\GGd)}
\ \asymp \
\sum_{e \subset [d]}\, \biggl(\sum_{k \in \ZZdpe} 
 \left\{2^{\alpha |k|_1 + \beta |k|_\infty} \omega_r^e(f,2^{-k})_{p,\GGd}\right\}^\theta \biggl)^{1/\theta} 
\end{equation}
with the corresponding change to sup when $\theta = \infty$. The notation $\Bab(\GGd)$ becomes explicitly reasonable if  we take the right side of \eqref{discretenorm[Bab]} as a definition of the quasi-norm of the space 
$\Bab(\GGd)$. Notice that  $B^{\alpha,0}_{p,\theta}(\GGd) = B^{{\alpha {\bf 1}}}_{p, \theta}(\GGd)$. However,  in general the space $B^{0,\beta}_{p,\theta}(\GGd)$ does not coincide with the classical isotropic Besov space $B^\beta_{p,\theta}(\GGd)$. This is a consequence of results in a forthcoming paper \cite{UV15}. We will need isotropic Besov spaces $B^\beta_{p,\theta}(\GGd)$ and introduce them separately below.

Let $\GGd$ be either $\IId$ or $\RRd$. We recall a notion of classical isotropic Besov space $\Bg(\GGd)$ and isotropic Sobolev space 
$W^\alpha_p(\GGd)$.  There are several different definitions with equivalent norms of these spaces. One can consult, for example, Chapters 4 and 5 in \cite{Ni75}, and Sections 2.3--2.5, 4.2, 4.4 in  \cite{Tr95}, where
the equivalence of these definitions are formulated in the form of a theorem on equivalence of corresponding different norms.  Below we will introduce one of them. We also refer the reader to the books \cite{BL76, Ni75, Tr95} for knowledge on these spaces. 

 Put for $r \in \NN$ and $h \in \RRd$
\[
\GGd(r,h)
:= \
\{x \in \GGd: \, x + rh \in \GGd\}.
\]
(Obviously, $\RRd(r,h) = \RRd$).
For functions $f$ on $\GGd$, the $r$th difference operator $\Delta_h^r$, $h \in \RRd$,  is defined by 
\begin{equation*}
\Delta_h^r(f,x) := \
\sum_{j =0}^r (-1)^{r - j} \binom{r}{j} f(x + jh), \ x \in \GGd(r,h),
\end{equation*}
and the $r$th modulus of smoothness $\omega_r(f,t)_{p,\GGd}$, $t \in \RR_+$, by
\begin{equation} \nonumber
\omega_r(f,t)_{p,\GGd}:= \sup_{|h| < t}\|\Delta_h^r(f)\|_{p,\GGd(r,h)}, 
\end{equation}
where $|h|:= \sqrt{|h_1|^2 + \cdots + |h_d|^2}$.
For $\alpha >0$,  $0 < p,\theta \le \infty$ and $r > \alpha$, we introduce the semi-norm 
$|f|_{\Bg(\GGd)}$ for functions $f \in L_p(\GGd)$ by
\begin{equation} \nonumber
|f|_{\Bg(\GGd)}:= \
\left\{
\begin{array}{cl}
\displaystyle
 \ \biggl(\int_{\II} \big\{
t^{-\alpha}\omega_r(f,t)_{p,\GGd} \big\}^ \theta t^{-1}dt \biggl)^{1/\theta}, 
& \theta < \infty, \\[2.5ex]
\displaystyle
 \sup_{t \in \II} \ t^{-\alpha}\omega_r(f,t)_{p,\GGd},  & \theta = \infty.
\end{array} \right.
\end{equation}
The isotropic Besov space $\Bg(\GGd)$ is defined as the set of all functions $f \in L_p(\GGd)$ for which the norm
\begin{equation} \nonumber
\|f\|_{\Bg(\GGd)}
:= \
\|f\|_{p,\GGd} + |f|_{\Bg(\GGd)}
\end{equation} 
is finite. Notice that the space $\Bg(\IId)$ can be seen as the quasi-normed space of restrictions of functions from $\Bg(\RRd)$ to $\IId$ equipped with the equivalent quasi-norm
\begin{equation}  \label{defB[restrictions]}
\|f\|_{\Bg(\IId)} 
:= \ 
\inf \left \{\|g\|_{\Bg(\RRd)}: g \in \Bg(\RRd), \ g|_{\IId} = f \right \}, 
\end{equation}
see Definition 4.2.1/1 and Theorem 4.4.2/2 in \cite{Tr95}.

Denote by $\mathcal{F}$ the Fourier transform in distributional sense for local integrable functions on $\RRd$. 
For $\alpha > 0$, $1 < p < \infty$, the Sobolev space $W^\alpha_p(\RRd)$ is defined as
 \begin{equation} \nonumber
W^\alpha_p(\RRd) 
:= \ 
\left \{f \in  L_p(\RRd) : \mathcal{F}^{-1}\left (1+ |y|^2 \right )^{\frac{\alpha}{2}}\mathcal{F}f \in L_p(\RRd) \right \} 
\end{equation}
equipped with the norm
\begin{equation} \nonumber
\|f\|_{W^\alpha_p(\RRd)} 
:= \ 
\|\mathcal{F}^{-1}\left (1+ |y|^2 \right )^{\frac{\alpha}{2}}\mathcal{F}f\|_{p,\RRd}. 
\end{equation}
The isotropic Sobolev space $W^\alpha_p(\IId)$ is defined as the normed space of restrictions of functions from $W^\alpha_p(\RRd)$ to $\IId$ equipped with the norm
\begin{equation} \nonumber
\|f\|_{W^\alpha_p(\IId)} 
:= \ 
\inf \left \{\|g\|_{W^\alpha_p(\RRd)}: g \in W^\alpha_p(\RRd), \ g|_{\IId} = f \right \}. 
\end{equation}

We need some quasi-norm equivalences for spaces 
$B^\alpha_{p, \theta}(\GGd)$.
If $i \in [d]$ and $t \in \RR_+$, the  $(r,i)$th partial modulus of smoothness $\omega_{r,i}(f,t)_{p,\GGd}$ is defined for functions $f$ on $\GGd$ by 
\begin{equation} \nonumber
\omega_{r,i}(f,t)_{p,\GGd}:= \sup_{|h| < t}\|\Delta_{i,h}^r(f)\|_{p,\GGd(i,r,h)}, 
\end{equation}
where $\GGd(i,r,h):= \{ x \in \GGd : x_i + rh \in \GG\}$ and 
\begin{equation*}
\Delta_h^{r,i}(f,x) := \
\sum_{j =0}^r (-1)^{r - j} \binom{r}{j} f(x_1,...,x_{i-1},x_i + jh,x_{i+1},...,x_d).
\end{equation*}
For $r > \alpha$, we introduce the semi-norm 
$|f|_{\Bg(\GGd)}$ for functions $f \in L_p(\GGd)$ by
\begin{equation} \nonumber
|f|_{B^\alpha_{p,\theta}(\GGd)_i}:= \
\left\{
\begin{array}{cl}
\displaystyle
 \ \biggl(\int_{\II} \big\{
t^{-\alpha}\omega_{r,i}(f,t)_{p,\GGd} \big\}^ \theta t^{-1}dt \biggl)^{1/\theta}, 
& \theta < \infty, \\[2.5ex]
\displaystyle
 \sup_{t \in \II} \ t^{-\alpha}\omega_{r,i}(f,t)_{p,\GGd},  & \theta = \infty.
\end{array} \right.
\end{equation}

 Let $1 \le  p, \theta \le \infty$ and $\alpha > 0$, $\GGd$ be either $\IId$ or $\RRd$. Then there holds the norm equivalence
\begin{equation} \label{NormB^beta}
\|f\|_{B^\alpha_{p, \theta}(\GGd)}
\ \asymp \ 
\|f\|_{p,\GG^d} +  
\sum_{i=1}^d |f|_{B^\alpha_{p,\theta}(\GGd)_i},
 \ \forall f \in {B^\alpha_{p, \theta}(\GGd)}.
\end{equation}
Indeed, for $\GGd=\RRd$, the right hand side of \eqref{NormB^beta} defines a norm of the classical anisotropic Besov space ${\bf B}^{(\alpha,...,\alpha)}_{p,\theta}(\RRd)$, see, e.g., \cite[Section 4.3.4]{Ni75}.
On the other hand, ${\bf B}^{(\alpha,...,\alpha)}_{p,\theta}(\RRd)$ coincides with  
$B^\alpha_{p,\theta}(\RRd)$ in the sense of norm equivalence \cite[Section 5.6.2]{Ni75}.
This proves \eqref{NormB^beta} for the case $\GGd=\RRd$. The case $\GGd=\IId$ of \eqref{NormB^beta} can be derived from the case $\GGd=\RRd$ and \eqref{defB[restrictions]}.

From \eqref{NormB^beta} and \eqref{|f|_BOe><} follows the norm equivalence
\begin{equation} \label{NormB^beta(5)}
\|f\|_{B^\alpha_{p, \theta}(\GGd)}
\ \asymp \ 
\|f\|_{p,\GG^d} +  
\sum_{i=1}^d \biggl(\sum_{k \in \ZZ_+}\left\{2^{\alpha k} \omega_{r,i}(f,2^{-k})_{p,\GGd}\right\}^\theta 
\biggl)^{1/\theta}, \ \forall f \in {B^\alpha_{p, \theta}(\GGd)}. 
\end{equation}

Let $\alpha > 0$, $1 < p < \infty$. Then there hold true the following inequalities for the norm of $W^\gamma_p(\GGd)$
\begin{equation} \label{ineq[normsW,B]}
\|f\|_{B^\gamma_{p,\max(p,2)}(\GGd)}
\ \ll \
\|f\|_{W^\alpha_p(\GGd)}
\ \ll \ 
 \|f\|_{B^\alpha_{p,\min(p,2)}(\GGd)}.
\end{equation}
These inequalities are a reformulated form of the embeddings
\[
B^\alpha_{p,\min(p,2)}(\GGd) \hookrightarrow  W^\alpha_p(\GGd) 
\hookrightarrow B^\gamma_{p,\max(p,2)}(\GGd)
\] 
which are a particular case of Theorem 4.6.1(b) in \cite{Tr95}.

Since in the present paper we consider only functions defined on $\IId$, for simplicity we somewhere drop the symbol $\IId$ in  the above notations. 

\subsection{Quasi-interpolation representations and quasi-norm equivalences}
We introduce quasi-interpolation operators for functions on $\IId$. For a given natural number $r,$ let  $M$ be the centered B-spline of order $r$ with support $[-r/2,r/2]$ and 
knots at the points $-r/2,-r/2 + 1,...,r/2 - 1, r/2$. 
Let $\Lambda = \{\lambda(j)\}_{j \in P(\mu)}$ be a given finite even sequence, i.e., 
$\lambda(-j) = \lambda(j),$ where $P(\mu):= \{j \in  \ZZ: \ |j| \le \mu \}$ 
and $\mu \ge r/2 - 1$. 
We define the linear operator $Q$ for functions $f$ on $\RR$ by  
\begin{equation} \label{def:Q}
Q(f,x):= \ \sum_{s \in \ZZ} \Lambda (f,s)M(x-s), 
\end{equation} 
where
\begin{equation} \label{def:Lambda}
\Lambda (f,s):= \ \sum_{j \in P(\mu)} \lambda (j) f(s-j).
\end{equation}
The operator $Q$ is local and bounded in $C(\RR)$,
 where $C(G)$ denotes the normed space of bounded continuous functions on $G$ with sup-norm $\|\cdot\|_{C(G)}$.
Moreover,
$
\|Q(f)\|_{C(\RR)} \le \|\Lambda \|\|f\|_{C(\RR)}  
$
for each $f \in C(\RR),$ where $\|\Lambda \|= \ \sum_{j \in P(\mu)} |\lambda (j)|$. 
An operator $Q$ of the form \eqref{def:Q}--\eqref{def:Lambda} reproducing 
$\Pp_{r-1}$, is called a {\it quasi-interpolation operator in} $C(\RR).$ For details on quasi-interpolation, see, e.g., \cite{Di11} and references there. 

We give some examples of quasi-interpolation operators. 
A piecewise linear quasi-interpolation operator is defined as
\begin{equation} \nonumber
Q(f,x):= \ \sum_{s \in \ZZ} f(s) M(x-s), 
\end{equation} 
where $M$ is the   
symmetric piecewise linear B-spline with support $[-1,1]$ and 
knots at the integer points $-1, 0, 1$.
It is related to the classical Faber-Schauder basis of the hat functions 
(see, e.g., \cite{Di11, Tr10}, for details). 
 A quadric quasi-interpolation operator is defined by
\begin{equation*} 
Q(f,x):= \ \sum_{s \in \ZZ} \frac {1}{8} \{- f(s-1) + 10f(s) - f(s+1)\} M(x-s), 
\end{equation*} 
where $M$ is the symmetric quadric B-spline with support $[-3/2,3/2]$ and 
knots at the half integer points $-3/2, -1/2, 1/2, 3/2$.
Another example is the cubic quasi-interpolation operator 
\begin{equation*} 
Q(f,x):= \ \sum_{s \in \ZZ} \frac {1}{6} \{- f(s-1) + 8f(s) - f(s+1)\} M(x-s), 
\end{equation*} 
where $M$ is the symmetric cubic B-spline with support $[-2,2]$ and 
knots at the integer points $-2, -1, 0, 1, 2$.

If $Q$ is a quasi-interpolation operator of the form 
\eqref{def:Q}--\eqref{def:Lambda}, for $h > 0$ and a function $f$ on $\RR$, 
we define the operator $Q(\cdot;h)$ by
$
Q(f;h) 
:= \ 
\sigma_h \circ Q \circ \sigma_{1/h}(f),
$
where $\sigma_h(f,x) = \ f(x/h)$.
From the definition it is easy to see that 
\begin{equation*}
Q(f,x;h)= \ 
\sum_{k}\Lambda (f,k;h)M(h^{-1}x-k),
\end{equation*}
where 
$
\Lambda (f,k;h):= 
 \ \sum_{j \in P(\mu)} \lambda (j) f(h(k-j)).
$

The operator $Q(\cdot;h)$ gives a good approximation for smooth functions \cite[p. 63--65]{BHR}.
However, it is not defined for 
a function $f$ on $\II$, and therefore, not appropriate for an 
approximate sampling recovery of $f$ from its sampled values at points in $\II$. 
An approach to construct a quasi-interpolation operator for functions on $\II$ is to extend it 
by interpolation Lagrange polynomials. This approach has been proposed in 
\cite{Di09} for the univariate case. Let us recall it.

Denote by $k_0$ the smallest integer such that $r \le 2^{k_0}$. For a non-negative integer $k,$ 
we put $x_j = j2^{-k-k_0},  j \in \ZZ.$ If $f$ is a function on $\II,$ let  
$U_k(f)$ and $V_k(f)$ 
be the $(r - 1)$th Lagrange polynomials interpolating $f$
at the $r$ left end points $x_0, x_1,..., x_{r-1},$ and 
$r$ right end points 
$x_{2^k - r + 1}, x_{2^k - r + 3},..., x_{2^k},$ 
of the interval $\II,$ respectively. 
The function $\bar{f}_k$ is defined as an extension of 
$f$ on $\RR$ by the formula
\begin{equation*} 
\bar{f}_k (x):= \
\begin{cases}
U_k(f,x), \ & x < 0, \\
f(x), \ & 0 \le x \le 1, \\
V_k(f,x), \ & x >1.
\end{cases}
\end{equation*} 
If $f$ is continuous on $\II$, then $\bar{f}_k$ is a continuous function on $\RR$ too. 
Let $Q$ be a quasi-interpolation operator of the form \eqref{def:Q}--\eqref{def:Lambda} in $C({\RR}).$ 
If $k \in \ZZ_+ $, we introduce the operator $Q_k$  by  
\begin{equation*}
Q_k(f,x) := \ Q(\bar{f}_k,x;2^{-k-k_0}), \  x \in \II, 
\end{equation*}
for a function $f$ on $\II$. 

We define the integer translated dilation $M_{k,s}$ of $M$ by   
\begin{equation*}
M_{k,s}(x):= \ M(2^{k+k_0} x - s), \ k \in {\ZZ}_+, \ s \in \ZZ. 
\end{equation*}
Then we have for $k \in {\ZZ}_+$,
\begin{equation} \nonumber
Q_k(f,x)  \ = \ 
\sum_{s \in J(k)} a_{k,s}(f)M_{k,s}(x), \ \forall x \in \II, 
\end{equation}
where 
$
J(k) := \ \{s \in \ZZ:\ -r/2 < s <   2^{k+k_0} + r/2 \}
$ 
is the set of $s$ for which $M_{k,s}$ 
do not vanish identically on  $\II,$ and the coefficient functional $a_{k,s}$ is defined by
\begin{equation*} 
a_{k,s}(f):= \ \Lambda(\bar{f}_k,s;{2^{-k}}) 
= \   
\sum_{|j| \le \mu} \lambda (j) \bar{f}_k(2^{-k}(s-j)).
\end{equation*}

For $k \in \ZZdp$, let the mixed operator $Q_k$ be defined by
\begin{equation} \label{def:Mixed[Q_k]} 
Q_k:= \prod_{i=1}^d  Q_{k_i},
\end{equation}
where the univariate operator
$Q_{k_i}$ is applied to the univariate function $f$ by considering $f$ as a 
function of  variable $x_i$ with the other variables held fixed.

We define the $d$-variable B-spline $M_{k,s}$ by
\begin{equation} \label{def:Mixed[M_{k,s}]}
M_{k,s}(x):=  \ \prod_{i=1}^d M_{k_i, s_i}( x_i),  
\ k \in {\ZZ}^d_+, \ s \in {\ZZ}^d.
\end{equation}
Then we have
\begin{equation*} 
Q_k(f,x)  \ = \ 
\sum_{s \in J^d(k)} a_{k,s}(f)M_{k,s}(x), \quad \forall x \in {\II}^d, 
\end{equation*}
where $M_{k,s}$ is the mixed B-spline  defined in \eqref{def:Mixed[M_{k,s}]}, 
$
J^d(k) := \ \{s \in {\ZZ}^d:\ \ - r/2 < s_i < 2^{k_i + k_0} + r/2, \ i \in [d]\}
$ 
is the set of $s$ for which $M_{k,s}$ do not vanish identically on  ${\II}^d$, 
\begin{equation} \label{def:Mixed[a_{k,s}(f)]}
a_{k,s}(f) 
\ := \   
a_{k_1,s_1}(a_{k_2,s_2}(...a_{k_d,s_d}(f))),
\end{equation}
and the univariate coefficient functional
$a_{k_i,s_i}$ is applied to the univariate function $f$ by considering $f$ as a 
function of  variable $x_i$ with the other variables held fixed. 

The operator $Q_k$ is a local bounded linear 
mapping in $C({\II}^d)$ for $r\ge 2$ and in $L_\infty(\IId)$ for $r=1$, and reproducing $\Pp_{r-1}^d$ the space 
of polynomials of order at most $r - 1$ in each variable $x_i$. 
In particular, we have for every $f \in C(\IId)$, 
\begin{equation} \label{ineq:Boundedness}
\|Q_k(f)\|_\infty  \le C \|\Lambda \|^d \|f\|_{C({\II}^d)}. 
\end{equation}

For $k \in \ZZdp$, we write  $k  \to  \infty$ if $k_i  \to  \infty$ for $i \in [d]$).

\begin{lemma}
We have for every $f \in C(\IId)$,
\begin{equation} \label{ineq[|f - Q_k(f)|])}
\|f - Q_k(f)\|_\infty
\ \le \
C \sum_{e \in [d], \ e \not= \varnothing} \omega_r^e(f,2^{-k})_\infty,
\end{equation}
and, consequently,
\begin{equation} \label{ConvergenceMixedQ_k(f)}
\|f - Q_k(f)\|_\infty  \to 0 , \ k  \to  \infty.
\end{equation}
\end{lemma}

\begin{proof}
For $d=1$, the inequality \eqref{ineq[|f - Q_k(f)|])} is of the form
\begin{equation} \label{ineq[|f - Q_k(f)|(d=1)])}
\|f - Q_k(f)\|_\infty
\ \le \
C \omega_r(f,2^{-k})_\infty.
\end{equation}
This inequality is derived from the inequalities (2.29)--(2.31) in \cite{Di11a} and the inequality 
\eqref{ineq:Boundedness}. For simplicity, let us prove the inequality \eqref{ineq[|f - Q_k(f)|])} for $d=2$ and 
$r \ge 2$. The general case can be proven in a similar way. Let $I$ be the identity operator and $k=(k_1,k_2)$. From 
the inequality \eqref{ineq[|f - Q_k(f)|(d=1)])} applied to $f$ as an univariate in each variable, we obtain 
\begin{equation} \nonumber
\begin{aligned}
\|f - Q_k(f)\|_\infty
\ &\le \
\|(I - Q_{k_1})(f)\|_\infty + \|(I - Q_{k_2})(f)\|_\infty + \|(I - Q_{k_1})(I - Q_{k_2})(f)\|_\infty\\
\ &\ll \
\omega_r^{\{1\}}(f,2^{-k})_\infty + \omega_r^{\{2\}}(f,2^{-k})_\infty + \omega_r^{[2]}(f,2^{-k})_\infty.
\end{aligned}
\end{equation}
\end{proof}

Further, we define 
the half integer translated dilation $M^*_{k,s}$ of $M$ by     
\begin{equation*}
M^*_{k,s}(x):= \ M(2^{k+k_0} x - s/2), \ k \in {\ZZ}_+, \ s \in \ZZ, 
\end{equation*}
and the $d$-variable B-spline $M^*_{k,s}$ by
\begin{equation*} 
M^*_{k,s}(x):= \ \ \prod_{i=1}^d M^*_{k_i, s_i}( x_i), 
\ k \in {\ZZ}^d_+, \ s \in {\ZZ}^d.
\end{equation*}
In what follows, the B-spline $M$ will be fixed. We will denote
$M^{(r)}_{k,s}:= M_{k,s}$ if the order $r$ of $M$ is even, 
and $M^{(r)}_{k,s}:= M^*_{k,s}$ if the order $r$ of $M$  is odd.
Let $J_r^d(k):= J^d(k)$ if $r$ is even, and 
\[
J_r^d(k):= \{s \in {\ZZ}^d: - r < s_i < 2^{k_i+k_0+1} + r, \ i \in [d] \}
\] 
if $r$ is odd. Notice that $J_r^d(k)$ is the set of $s$ for which 
$M^{(r)}_{k,s}$ do not vanish identically on  ${\II}^d$. 
Denote by $\Sigma_r^d(k)$ the span of the B-splines 
$M^{(r)}_{k,s}, \ s \in J_r^d(k)$.
If $0 < p \le \infty,$ 
for all $k \in {\ZZ}^d_+$ and all $g \in \Sigma_r^d(k)$ such that
\begin{equation} \label{def:StabIneq}
g = \sum_{s \in J_r^d(k)} a_s M^{(r)}_{k,s},
\end{equation}
there is the quasi-norm equivalence 
\begin{equation} \label{eq:StabIneq}
 \|g\|_p 
\ \asymp \ 2^{- |k|_1/p}\|\{a_s\}\|_{p,k},
\end{equation}
where
\begin{equation*}
\| \{a_s\} \|_{p,k}
 \ := \
\biggl(\sum_{s \in J_r^d(k)}| a_s |^p \biggl)^{1/p} 
\end{equation*}
with the corresponding change when $p= \infty$. 

For convenience we define the univariate operator $Q_{-1}$ by putting $Q_{-1}(f)=0$ for all $f$ on $\II$. Let  the operator $q_k$, $k \in {\ZZ}^d_+$, be defined 
in the manner of the definition \eqref{def:Mixed[Q_k]} by
\begin{equation} \label{eq:Def[q_k]}
q_k\ := \ \prod_{i=1}^d \left(Q_{k_i}- Q_{k_i-1}\right). 
\end{equation}
We have 
 \begin{equation} \label{eq:MixedQ_k(2)}
Q_k
\ = \ 
\sum_{k' \le k}q_{k'}. 
\end{equation}
From \eqref{eq:MixedQ_k(2)} and \eqref{ConvergenceMixedQ_k(f)} it is easy to see that 
a continuous function  $f$ has the decomposition
\begin{equation} \nonumber
f \ =  \ \sum_{k \in {\ZZ}^d_+} q_k(f)
\end{equation}
with the convergence in the norm of $L_\infty(\IId)$.  
By using the definition of \eqref{eq:Def[q_k]} and the refinement equation for the B-spline $M$, we can represent 
the component functions $q_k(f)$ as 
 \begin{equation} \label{eq:RepresentationMixedq_k(f)}
q_k(f) 
= \ \sum_{s \in J_r^d(k)}c^{(r)}_{k,s}(f)M^{(r)}_{k,s},
\end{equation}
where $c^{(r)}_{k,s}$ are certain coefficient functionals of 
$f,$ which are defined as follows (see \cite{Di11} for details). 
We first define $c^{(r)}_{k,s}$ for univariate functions ($d=1$). If the order $r$ of the B-spline $M$ is even,
\begin{equation}\label{def:c_{k,s}}
c^{(r)}_{k,s}(f)
\ := \
a_{k,s}(f) - a^\prime_{k,s}(f), \ k \ge 0,
\end{equation}
where
\begin{equation*}
a^\prime_{k,s}(f):= \ 2^{-r+1} \sum_{(m,j) \in C_r(k,s)}
 \binom{r}{j} a_{k-1,m}(f),  \ k > 0, \ \ 
a^\prime_{0,s}(f):= 0.
\end{equation*}
 and 
\begin{equation*} 
C_r(k,s) := \{(m,j): 2m + j - r/2 = s, \ m \in J(k-1), \  0 \le j \le r\}, \ k > 0, \ \
C_r(0,s) := \{0\}.
\end{equation*}
If the order $r$ of the B-spline $M$ is odd,  
\begin{equation*}
c^{(r)}_{k,s}(f)
\ := \
\begin{cases}
0, & \ k = 0, \\
a_{k,s/2}(f), & \ k >0, \ s \ \text{even}, \\
 2^{-r+1} \sum_{(m,j) \in C_r(k,s)}
 \binom{r}{j} a_{k-1,m}(f), & \ k >0, \ s \ \text{odd}, 
\end{cases}
\end{equation*}
where
 \begin{equation*} 
C_r(k,s) := \{(m,j): 4m + 2j - r = s, \ m \in J(k-1), \  0 \le j \le r\}, \ k > 0, \ \
C_r(0,s) := \{0\}.
\end{equation*}
 
In the multivariate case, the representation  \eqref{eq:RepresentationMixedq_k(f)} holds true 
with the $c^{(r)}_{k,s}$ which are defined in the manner of the definition of 
\eqref{def:Mixed[a_{k,s}(f)]} by
\begin{equation} \label{def:Mixedc_{k,s}}
c^{(r)}_{k,s}(f) 
\ = \   
c^{(r)}_{k_1,s_1}((c^{(r)}_{k_2,s_2}(... c^{(r)}_{k_d,s_d}(f))).
\end{equation}

Thus, we have proven  the following  
\begin{lemma} \label{lemma[representation]}
Every continuous function $f$ on $\IId$ can be represented as B-spline series 
\begin{equation} \label{eq:B-splineRepresentation}
f \ = \sum_{k \in {\ZZ}^d_+} \ q_k(f) = 
\sum_{k \in {\ZZ}^d_+} \sum_{s \in J_r^d(k)} c^{(r)}_{k,s}(f)M^{(r)}_{k,s}, 
\end{equation}  
converging in the norm of $L_\infty(\IId)$, where the coefficient functionals $c^{(r)}_{k,s}(f)$ are explicitly constructed by formula 
\eqref{def:c_{k,s}}--\eqref{def:Mixedc_{k,s}} as
linear combinations of at most $N$ function
values of $f$ for some $N \in \NN$ which is independent of $k,s$ and $f$.
\end{lemma}

We prove now theorems on quasi-interpolation  representation of functions from $\BO$ and $\Ba$, $\Bab$ by series 
\eqref{eq:B-splineRepresentation} satisfying a discrete equivalent quasi-norm. We need some auxiliary lemmas.
Let us use the notation: $x_+:= ((x_1)_+, ..., (x_d)_+)$ for $x \in \RR^d$. Put 
${\NN}^d(e):= \{s \in \ZZdp: s_i > 0 , \ i \in e, \ s_i = 0 ,  i \notin e\}$ for $e \subset [d]$ 
(in particular, ${\NN}^d(\varnothing)= \{0\}$ and ${\NN}^d([d])= {\NN}^d$).
 We have ${\NN}^d(u)\cap {\NN}^d(v) = \varnothing$ if $u \ne v$, 
 and the decomposition
$
{\ZZ}^d_+ =  \bigcup_{e \subset [d]} {\NN}^d(e).
$

\begin{lemma}[\cite{Di11}] \label{Lemma:IneqMixedq_k(f))} 
Let $0 < p \le \infty$ and $0 < \tau \le \min (p,1)$. 
Then for any $f \in C(\II^d)$ and $k \in {\NN}^d(e)$, we have 
\begin{equation*}  
\| q_k(f)\|_p  
\ \le \ 
C  \sum_{v \supset e} \left( \sum_{s \in {\ZZ}^d_+(v), \ s \ge k}
 \left\{ 2^{|s-k|_1/p}\omega_{r}^v(f,2^{-s})_p \right\}^\tau \right)^{1/\tau}
\end{equation*}
with some constant $C$ depending at most on $r, \mu, p, d$ and $\|\Lambda\|,$  
whenever the sum in the right-hand side is finite.
\end{lemma}

The following lemma can be proven in a way similar to the proof of \cite[Lemma 2.3]{Di11}.

\begin{lemma} \label{Lemma:IneqMixedomega_{r}^e}  
Let $0 < p \le \infty, \ 0 < \tau \le \min(p,1), \ \delta = \min (r, r - 1 + 1/p)$. 
Let $g \in L_p(\IId)$ be represented by  the  series 
\begin{equation*} 
g \ = \sum_{k \in {\ZZ}^d_+} \ g_k, \ g_k \in \Sigma_r^d(k)
\end{equation*}
converging in the norm of $L_\infty(\IId)$.
Then for any $k \in {\ZZ}^d_+(e)$, there holds the inequality 
\begin{equation*} 
\omega_{r}^e(g,2^{-k})_p 
\ \le \ 
C   \left( \sum_{s\in {\ZZ}^d_+}
 \left\{ 2^{-\delta |(k-s)_+|_1} \| g_s\|_p \right\}^\tau \right)^{1/\tau}
\end{equation*}
with some constant $C$ depending at most on $r, \mu, p, d$ and $\|\Lambda\|,$  
whenever the sum on the right-hand side is finite.
\end{lemma}

For $0 < p \le \infty$, $k \in \ZZ_+$ and $i \in [d]$, denote by $g \in \Sigma_{r,i}^d(k)_p$ the set of all functions 
$g \in L_p(\IId)$ such that
\begin{equation} \nonumber
g(x) = \sum_{s \in J_r^1(k)} a_s(x_1,...,x_{i-1},x_{i+1},...,x_d) \, M^{(r)}_{k,s}(x_i).
\end{equation}

\begin{lemma} \label{Lemma:IneqMixedomega_{r,i}}  
Let $0 < p \le \infty, \ 0 < \tau \le \min(p,1), \ \delta = \min (r, r - 1 + 1/p)$. 
Let $g \in L_p(\IId)$ be represented by  the  series 
\begin{equation} \label{representationby[Sigma_{r,i}^d(k)_p]}
g \ = \sum_{s \in \ZZ_+} \ g_s, \ g_s \in \Sigma_{r,i}^d(s)_p
\end{equation}
converging in the norm of $L_\infty(\IId)$.
Then for any $k \in \ZZ_+$, there holds the inequality 
\begin{equation*} 
\omega_{r,i}(g,2^{-k})_p 
\ \le \ 
C   \left( \sum_{s\in \ZZ_+}
 \left\{ 2^{-\delta (k-s)_+} \| g_s\|_p \right\}^\tau \right)^{1/\tau}
\end{equation*}
with some constant $C$ depending at most on $r, \mu, p, d$ and $\|\Lambda\|,$  
whenever the sum on the right-hand side is finite.
\end{lemma}

\begin{proof}
This lemma can be also proven in a way similar to the proof of \cite[Lemma 2.3]{Di11} with a slight modification by replacing $\Delta^r_h$, $h \in \IId$, and the representation \eqref{eq:B-splineRepresentation} with 
$\Delta^{r,i}_h$, $h \in \II$, and the representation \eqref{representationby[Sigma_{r,i}^d(k)_p]}.
\end{proof}

Let $0 < \theta \le \infty$ and $\psi: \ZZdp \to \RR$. 
If $\{g_k\}_{k \in \ZZdp}$ is a sequence of numbers, we define the ``quasi-norm" $\|\{ g_k\}\|_{b_{\theta}^{\psi}}$ by
\begin{equation*}
\|\{ g_k\}\|_{b_{\theta}^{\psi}}
\ := \ 
 \biggl(\sum_{k \in \ZZdp} 
 \left(2^{\psi(k)} |g_k|\right)^\theta \biggl)^{1/\theta} 
\end{equation*}
with the usual change to a supremum when $\theta = \infty.$ 
We will need the following generalized discrete Hardy inequality 
(see, e.g, \cite{DL} for the univariate case with $\psi(k)= \alpha k, \ \alpha >0$).

\begin{lemma} \label{lemma[HardyIneq]}
Let $\{a_k\}_{k \in \ZZdp}$ and $\{b_k\}_{k \in \ZZdp}$ be two positive 
sequences and let for some $A > 0, \ \tau > 0, \ \delta > 0$ 
\begin{equation} \label{HardyIneq(1)}
b_k
 \  \le \ A \biggl(  \sum_{s \in \ZZdp} 
\left(2^{- \delta |(k-s)_+|_1}a_s\right)^\tau \biggl)^{1/\tau}.
\end{equation}
Let the function $\psi: \ZZdp \to \RR$ satisfy the following.
There are numbers $c_1, c_2 \in \RR
$, $\epsilon > 0$ and $0 < \zeta < \delta$ such that
\begin{equation}  \label{HardyIneq(conditionI)}
\psi(k)- \epsilon |k|_1
 \  \le \ 
\psi(k')- \epsilon |k'|_1 + c_1, \ \ k \le k', \ k, k' \in \ZZdp,
\end{equation}
and
\begin{equation} \label{HardyIneq(conditionII)}
\psi(k)- \zeta |k|_1
 \  \ge \ 
\psi(k')- \zeta |k'|_1 + c_2, \ \ k \le k', \ k, k' \in \ZZdp.
\end{equation}
Then for $0 < \theta \le \infty$, we have 
\begin{equation} \label{HardyIneq(2)}
\ \|\{ b_k\}\|_{b_\theta^\psi}
\ \le \ 
 C A \|\{ a_k\}\|_{b_\theta^\psi}
\end{equation}
with $C = C(c_1,c_2, \epsilon, \delta, \theta, d) >0$. 
\end{lemma}
\begin{proof}
Because the right side of \eqref{HardyIneq(1)} becomes larger when $\tau$ becomes smaller, we can assume 
$\tau < \theta$. For $e \subset [d]$ and  $s \in \ZZ^d$, let  $\bar{e}: = [d] \setminus e$ and $s(e)\in \ZZ^d$ be defined by $s(e)_j = s_j$ if  $j \in e$, and $s(e)_j = 0$ if $j \in \bar{e}$. From \eqref{HardyIneq(1)} we have
\begin{equation} \label{ineq[b_k]}
b_k
 \  \ll \ 
A \sum_{e \subset [d]} B_k(e), \ k \in \ZZdp,
\end{equation}
where
\begin{equation} \nonumber
B_k(e)
:= \ 
2^{-\delta |k(e)|_1}\biggl(\sum_{s \in Z(e,k)} 
\left(2^{\delta |s(e)|_1}a_s\right)^\tau \biggl)^{1/\tau}
\end{equation}
and
$
Z(e,k)
:= \ 
\{s \in \ZZdp: \, s_j \le k_j, \, j\in e; \, s_j > k_j, j \not\in e\}.
$
Take numbers $\epsilon', \zeta', \theta'$ with the conditions 
$0 < \epsilon' <\epsilon$, $\zeta < \zeta' < \delta$ and $\tau/\theta + \tau/\theta' = 1$, respectively. 
Applying H\"older's inequality with exponents $\theta/\tau, \theta'/\tau$, we obtain
\begin{equation} \nonumber
\begin{aligned}
B_k(e)
\ &\le \
2^{-\delta |k(e)|_1}\biggl(\sum_{s \in Z(e,k)} 
\big(2^{\zeta'|s(e)|_1 + \epsilon'|s(\bar{e})|_1} a_s\big)^\theta \biggl)^{1/\theta}
\biggl(\sum_{s \in Z(e,k)} 
\big(2^{(\delta-\zeta')|s(e)|_1 - \epsilon'|s(\bar{e})|_1}\big)^{\theta'} \biggl)^{1/\theta'} \\
\ &\ll \
2^{-\delta |k(e)|_1}\biggl(\sum_{s \in Z(e,k)} 
\big(2^{\zeta'|s(e)|_1 + \epsilon'|s(\bar{e})|_1} a_s\big)^\theta \biggl)^{1/\theta}
2^{(\delta-\zeta')|k(e)|_1 - \epsilon'|k(\bar{e})|_1} \\
\ &\ll \
2^{-\zeta'|k(e)|_1 - \epsilon'|k(\bar{e})|_1} 
\biggl(\sum_{s \in Z(e,k)} 
\big(2^{\zeta'|s(e)|_1 + \epsilon'|s(\bar{e})|_1} a_s\big)^\theta \biggl)^{1/\theta}. \\
\end{aligned}
\end{equation}
Hence,
\begin{equation} \label{inq[|{B_k(e)}|]}
\begin{aligned}
\|\{B_k(e)\}\|_{b_\theta^\psi}^\theta
\ &\ll \
\sum_{k \in \ZZdp} 2^{\theta(\psi(k)-\zeta'|k(e)|_1 - \epsilon'|k(\bar{e})|_1)} 
\sum_{s \in Z(e,k)} 
\left(2^{\zeta'|s(e)|_1 + \epsilon'|s(\bar{e})|_1} a_s\right)^\theta  \\
\ &\ll \
\sum_{s \in \ZZdp}
2^{\theta(\zeta'|s(e)|_1 + \epsilon'|s(\bar{e})|_1)} a_s^\theta 
 \sum_{k \in X(e,s)} 
2^{\theta(\psi(k)-\zeta'|k(e)|_1 - \epsilon'|k(\bar{e})|_1)},
\end{aligned}
\end{equation}
where 
\[
X(e,s):= \{k \in \ZZdp: \, k_j \ge s_j, \, j\in e; \, k_j < s_j, j \not\in e\}.
\]
From \eqref{HardyIneq(conditionI)} and \eqref{HardyIneq(conditionII)} we can easily derive  for $k \in X(e,s)$,
\[
\psi(k)
\le
\psi(s(e) + k(\bar{e}))- \zeta |s(e)|_1  + \zeta|k(e)|_1, 
\]
and
\[
\psi(s(e) + k(\bar{e})) \le
\psi(s)-  \epsilon |s(\bar{e})|_1 +  \epsilon  |k(\bar{e})|_1.
\]
Consequently,
\[\psi(k)-\zeta'|k(e)|_1 - \epsilon'|k(\bar{e})|_1
\ \le \
\psi(s)-\zeta|s(e)|_1 - \epsilon|s(\bar{e})|_1 - (\zeta' - \zeta)|k(e)|_1 + (\epsilon - \epsilon')|k(\bar{e})|_1, 
\]
and therefore, we can continue the estimation \eqref{inq[|{B_k(e)}|]} as
\begin{equation} \nonumber
\begin{aligned}
\|\{B_k(e)\}\|_{b_\theta^\psi}^\theta
\ &\ll \
\sum_{s \in \ZZdp}
2^{\theta(\psi(s) + (\zeta'-\zeta)|s(e)|_1 - (\epsilon-\epsilon')|s(\bar{e})|_1} a_s^\theta 
 \sum_{k \in X(e,s)} 
2^{\theta(- (\zeta'-\zeta)|k(e)|_1 + (\epsilon-\epsilon')|k(\bar{e})|_1} \\
\ &\ll \
\sum_{s \in \ZZdp}
2^{\theta(\psi(s) + (\zeta'-\zeta)|s(e)|_1 - (\epsilon-\epsilon')|s(\bar{e})|_1} a_s^\theta \, 
 2^{\theta(- (\zeta'-\zeta)|s(e)|_1 + (\epsilon-\epsilon')|s(\bar{e})|_1} \\
\ &= \
\sum_{s \in \ZZdp}
2^{\theta \psi(s)} a_s^\theta \ = \ \|\{a_k\}\|_{b_\theta^\psi}^\theta.
\end{aligned}
\end{equation}
Hence, by \eqref{ineq[b_k]} we prove \eqref{HardyIneq(2)}.
\end{proof}

We now are able to prove quasi-interpolation  B-spline representation theorems for functions from $\BO$ and $\Bab$, $\Ba$. 
For functions $f$ on ${\II}^d$, we introduce the following quasi-norms:
\begin{equation*} 
\begin{aligned}
B_2(f)
 \ & := \ 
\biggl(\sum_{k \in \ZZdp} 
\left\{\|q_k(f)\|_p/\Omega(2^{-k}) \right\}^\theta \biggl)^{1/\theta}; \\
B_3(f)
\ & := \ 
\biggl(\sum_{k \in \ZZdp} 
\left\{ 2^{ - |k|_1/p}\|\{c^{(r)}_{k,s}(f)\}\|_{p,k}/\Omega(2^{-k}) \right\}^\theta \biggl)^{1/\theta}.
\end{aligned}
\end{equation*}
Observe that by  \eqref{eq:StabIneq} 
the quasi-norms $B_2(f)$ and $B_3(f)$ are equivalent.

\begin{theorem} \label{Theorem:Representation}
Let $\ 0 < p, \theta \le \infty$ and $\Omega$ satisfy the additional conditions:
there are  numbers $ \mu, \rho > 0$ and $C_1, C_2 > 0$ such that
\begin{equation} \label{condition(a)}
\Omega(t) \,  \prod_{i=1}^d t_i^{-\mu}
\ \le \ 
C_1 \Omega(t') \, \prod_{i=1}^d {t'}_i^{-\mu} , \ \ t \le t', \ t, t' \in \IId, 
\end{equation} 
\begin{equation} \label{condition(b)}
\Omega(t) \,  \prod_{i=1}^d t_i^{-\rho}
\ \ge \ 
C_2 \Omega(t') \, \prod_{i=1}^d {t'}_i^{-\rho} , \ \ t \le t', \ t, t' \in \IId.  
\end{equation}
Then we have the following. 
\begin{itemize}

\item[(i)] If $\mu > 1/p$ and $\rho < r$, then a function $f \in \BO$ can be represented by the B-spline series 
\eqref{eq:B-splineRepresentation}
satisfying the convergence condition
\begin{equation} \label{ConvegenceCondition[BO]}
B_2(f) \ \ll \ \|f\|_{\BO}.
\end{equation}

\item[(ii)]  If $ \rho < \min(r, r - 1 + 1/p)$,
then a function $g$ on ${\II}^d$ represented by a series 
\begin{equation} \label{series[g]}
g \ = \sum_{k \in {\ZZ}^d_+} \ g_k = 
\sum_{k \in {\ZZ}^d_+} \sum_{s \in J_r^d(k)} c_{k,s} M^{(r)}_{k,s}, 
\end{equation}
satisfying the condition
\begin{equation} \nonumber
B_4(g):= \  
\biggl(\sum_{k \in \ZZdp} \left\{\|g_k\|_p/\Omega(2^{-k}) \right\}^\theta \biggl)^{1/\theta} \ < \ \infty,
\end{equation}
belongs  
the  space $\BO$.
Moreover, 
$
\|g\|_{\BO}
\ \ll \ 
B_4(g).
$

\item[(iii)]  If $\mu > 1/p$ and $ \rho < \min(r, r - 1 + 1/p)$,
then a function $f$ on ${\II}^d$ belongs to 
the  space $\BO$ if and only if 
$f$ can be represented by the series \eqref{eq:B-splineRepresentation} satisfying the convergence condition
\eqref{ConvegenceCondition[BO]}.
Moreover, the  quasi-norm $\|f\|_{\BO}$ is equivalent to  the quasi-norm $B_2(f)$.
\end{itemize} 
\end{theorem}

\begin{proof} Put $\phi(x): =  \log_2 [1/\Omega (2^{-x})]$. 
Due to \eqref{condition(a)}--\eqref{condition(b)}, the function $\phi$ satisfies the following conditions 
\begin{equation} \label{condition(a')}
\phi (x) - \mu |x|_1
\ \le \ 
 \phi (x') - \mu |x'|_1 + \log_2 C_1, \ \ x \le x', \ x, x' \in \RRdp, 
\end{equation} 
and
\begin{equation} \label{condition(b')}
\phi (x) - \rho|x|_1
\ \ge \ 
\phi (x') - \rho|x'|_1 + \log_2 C_2, \ \ x \le x', \ x, x' \in \RRdp.   
\end{equation}
We also have
\begin{equation} \label{[B_1(f)=]}
B_1(f)
\ = \
\sum_{e \subset [d]}\, \biggl(\sum_{k \in \ZZdpe} 
 \left\{2^{\phi(k)}\omega_r^e(f,2^{-k})_p)\right\}^\theta \biggl)^{1/\theta}, 
\end{equation}
with the corresponding change to sup when $\theta = \infty$.
Fix a number $0 < \tau \le \min (p,1)$.

\noindent
{\it Assertion} (i):  From \eqref{condition(a')} we derive 
$\mu |k|_1 \le  \phi(k) + c,  \ k \in \ZZdp$, 
for some constant $c$. Hence, by Lemma \eqref{normequivalence[B_1]} and \eqref{[B_1(f)=]} we have 
$
\|f\|_{B^{\mu {\bf 1}}_p}
\ \le \ 
C\|f\|_{\BO}, \ f \in \BO,
$
for some constant $C$. Since for $\mu > 1/p$, $B^{\mu {\bf 1}}_p$ is compactly embedded into $C(\IId)$, the same holds for $\BO$. Take an arbitrary $f \in \BO$. Then $f$ can be treated as an element in $C(\IId)$.  
By Lemma 
\ref{lemma[representation]}  $f$ is represented as B-spline series 
\eqref{eq:B-splineRepresentation} converging in the norm of $L_\infty(\IId)$.
For $k \in {\ZZ}^d_+$, put 
\begin{equation*}
b_k := 2^{|k|_1/p} \|q_k (f)\|_p, \ \
 a_k 
\ := \ 
\biggl( \sum_{v \supset e}
 \left\{ 2^{|k|_1/p}\omega_{r}^v(f,2^{-k})_p \right\}^\tau \biggl)^{1/\tau}
\end{equation*}
 if $k \in {\NN}^d(e)$.
 By Lemma \ref{Lemma:IneqMixedq_k(f))} we have for $k \in {\ZZ}^d_+$,
\begin{equation*}
b_k 
\ \le \ 
C   \biggl( \sum_{s \ge k}^\infty  a_s ^{\tau} \biggl)^{1/\tau}
\ \le \
C \biggl(  \sum_{s \in {\ZZ}^d_+}\left(2^{- \delta |(k-s)_+|_1}a_s\right)^\tau \biggl)^{1/\tau},\ k \in \ZZdp,
\end{equation*}
for a fixed $\delta > \rho + 1/p$.
Let the function $\psi$ be defined by $\psi(k) = \phi(k) - |k|_1/p, \ k \in \ZZdp$. 
By the inequality $\mu > 1/p$, \eqref{condition(a')} and \eqref{condition(b')}, it is easy to see that 
\begin{equation} \nonumber
\psi(k)- \epsilon |k|_1
 \le  
\psi(k')- \epsilon |k'|_1 + \log_2 C_1,  \ k \le k', \ k, k' \in \ZZdp,
\end{equation}
and
\begin{equation} \nonumber
\psi (k) - \zeta|k|_1
\ge  
\psi (k') - \zeta|k'|_1 + \log_2 C_2,  \ k \le k', \ k, k' \in \ZZdp,   
\end{equation}
for $\epsilon < \mu - 1/p$ and $\zeta = \rho + 1/p$. Hence, applying Lemma \ref{lemma[HardyIneq]} gives 
\begin{equation*} 
B_2(f) \ = \ \|\{ b_k\}\|_{b_\theta^\psi}
\ \le \ 
 C \|\{ a_k\}\|_{b_\theta^\psi} \ \asymp \ B_1(f) \ \asymp \ \|f\|_{\BO}.
\end{equation*}

\noindent
{\it Assertion} (ii): 
For $k \in {\ZZ}^d_+$, define 
\begin{equation*}
b_k := \left( \sum_{v \supset e}
 \left\{ \omega_{r}^v(g,2^{-k})_p \right\}^\tau \right)^{1/\tau}, \ \
 a_k 
\ := \ 
\|g_k\|_p
\end{equation*}
if $k \in {\NN}^d(e)$. By Lemma \ref{Lemma:IneqMixedomega_{r}^e} we have
for any  $k \in {\ZZ}^d_+(e)$,
\begin{equation*} 
\omega_{r}^e(g,2^{-k})_p 
\ \le \ 
C_3   \biggl( \sum_{s\in {\ZZ}^d_+}
 \left\{ 2^{-\delta |(k-s)_+|_1} \| g_s\|_p \right\}^\tau \biggl)^{1/\tau},
\end{equation*}
where $\delta = \min(r, r - 1 + 1/p)$.
Therefore,
\begin{equation*} 
b_k
 \  \le \ C_4 \biggl(  \sum_{s \in {\ZZ}^d_+} 
\left(2^{- \delta |(k-s)_+|_1}a_s\right)^\tau \biggl)^{1/\tau},\ k \in \ZZdp.
\end{equation*}
Taking  $\zeta = \rho$ and $0 < \epsilon < \mu$, we obtain by \eqref{condition(a')} and \eqref{condition(b')}
\begin{equation} \nonumber
\phi(k)- \epsilon |k|_1
 \  \le \ 
\phi(k')- \epsilon |k'|_1 + \log_2 C_1, \ \ k \le k', \ k, k' \in \ZZdp,
\end{equation}
and
\begin{equation} \nonumber
\phi(k)- \zeta |k|_1
 \  \ge \ 
\phi(k')- \zeta |k'|_1 + \log_2 C_2, \ \ k \le k', \ k, k' \in \ZZdp.
\end{equation}
Applying Lemma \ref{lemma[HardyIneq]} we get
\begin{equation*} 
\|g\|_{\BO} \ \asymp \ B_1(g) \ \asymp \ \|\{ b_k\}\|_{b_\theta^\phi}
\ \le \ 
 C \|\{ a_k\}\|_{b_\theta^\phi} \ = \ B_4(g).
\end{equation*}
Assertion (ii) is proven.

\noindent
{\it Assertion} (iii): This assertion follows from Assertions (i) and (ii).
\end{proof}

From  Assertion (ii) in Theorem \ref{Theorem:Representation} we obtain

\begin{corollary} \label{Corollary[B-inequality]}
Let $\ 0 < p, \theta \le \infty$ and $\Omega$ satisfy the assumptions of Assertion (ii) in Theorem 
\ref{Theorem:Representation}.
Then for every $k \in \ZZdp$, we have
\begin{equation} \nonumber
\|g\|_{\BO}
\ \ll \
\|g\|_p/\Omega(2^{-k}), \ g \in \Sigma_r^d(k).
\end{equation} 
\end{corollary}

\begin{theorem} \label{Theorem:Representation[Ba]}
Let $\ 0 < p, \theta \le \infty$ and  $a \in \RRdp$. 
Then we have the following. 
\begin{itemize}
\item[(i)] If $1/p < \min_{j \in [d]} a_j \le \max_{j \in [d]} a_j < r$, then a function $f \in \Ba$ can be represented by the mixed B-spline series \eqref{eq:B-splineRepresentation}
satisfying the convergence condition
\begin{equation} \label{ConvegenceCondition[Ba]}
B_2(f) 
\ = \ 
\left( \sum_{k \in {\ZZ}^d_+} \ \{2^{(a,k)}\|q_k(f)\|_p\}^\theta\right)^{1/\theta} 
\ \ll \ 
\|f\|_{\Ba}.
\end{equation}

\item[(ii)]  If  $0 < \min_{j \in [d]} a_j \le \max_{j \in [d]} a_j < \min(r, r - 1 + 1/p)$,
then a function $g$ on ${\II}^d$ represented by a series \eqref{series[g]}
satisfying the condition
\begin{equation} \nonumber
B_4(g):= \  
\biggl(\sum_{k \in \ZZdp} \left\{2^{(a,k)}\|g_k\|_p \right\}^\theta \biggl)^{1/\theta} \ < \ \infty,
\end{equation}
belongs  
the  space $\Ba$.
Moreover, 
$
\|g\|_{\Ba}
\ \ll \ 
B_4(g).
$

\item[(iii)]  If $1/p < \min_{j \in [d]} a_j \le \max_{j \in [d]} a_j < \min(r, r - 1 + 1/p)$,
then a function $f$ on ${\II}^d$ belongs to 
the space $\Ba$ if and only if 
$f$ can be represented by the series \eqref{eq:B-splineRepresentation} satisfying the convergence condition
\eqref{ConvegenceCondition[Ba]}.
Moreover, the  quasi-norm $\|f\|_{\Ba}$ is equivalent to  the quasi-norms $B_2(f)$.
\end{itemize} 
\end{theorem}

\begin{proof} For $\Omega$ as in \eqref{def[Ba]}, we have
$1/\Omega (2^{-x}) \ = \ 2^{(a,x)},\ \  x \in \RRdp$. 
One can directly verify the conditions \eqref{conditionI}--\eqref{conditionIII}  and the conditions \eqref{condition(a)}--\eqref{condition(b)} with 
$1/p < \mu < \min_{j \in [d]} a_j$ and $\rho = \max_{j \in [d]} a_j$, for $\Omega$ defined in \eqref {def[Ba]}.
Applying Theorem \ref{Theorem:Representation}(i), we obtain the assertion (i). 

The assertion (ii) can be proven in a similar way. The assertion (iii) follows from the assertions (i) and (iii).
\end{proof}

\begin{theorem} \label{Theorem:Representation[Bab]}
Let $\ 0 < p, \theta \le \infty$ and $\alpha \in \RR_+$, $\beta \in \RR$. 
Then we have the following. 
\begin{itemize}
\item[(i)] If $1/p < \min (\alpha,\alpha + \beta) \le  \max(\alpha,\alpha + \beta) < r$, then a function $f \in \Bab$ can be represented by the mixed B-spline series \eqref{eq:B-splineRepresentation}
satisfying the convergence condition
\begin{equation} \label{ConvegenceCondition[Bab]}
B_2(f) 
\ = \ 
\left( \sum_{k \in {\ZZ}^d_+} \ \{2^{\alpha |k|_1 + \beta |k|_\infty}\|q_k(f)\|_p\}^\theta\right)^{1/\theta} 
\ \ll \ 
\|f\|_{\Bab}.
\end{equation}

\item[(ii)]  If  $0 < \min (\alpha,\alpha + \beta) \le \max(\alpha,\alpha + \beta) < \min(r, r - 1 + 1/p)$,
then a function $g$ on ${\II}^d$ represented by a series \eqref{series[g]}
satisfying the condition
\begin{equation} \nonumber
B_4(g):= \  
\biggl(\sum_{k \in \ZZdp} \left\{2^{\alpha |k|_1 + \beta |k|_\infty}\|g_k\|_p \right\}^\theta \biggl)^{1/\theta} \ < \ \infty,
\end{equation}
belongs  
the  space $\Bab$.
Moreover, 
$
\|g\|_{\Bab}
\ \ll \ 
B_4(g).
$

\item[(iii)]  If  $1/p < \min (\alpha,\alpha + \beta) \le \max(\alpha,\alpha + \beta) < \min(r, r - 1 + 1/p)$,
then a function $f$ on ${\II}^d$ belongs to 
the space $\Bab$ if and only if 
$f$ can be represented by the series \eqref{eq:B-splineRepresentation} satisfying the convergence condition
\eqref{ConvegenceCondition[Bab]}.
Moreover, the  quasi-norm $\|f\|_{\Bab}$ is equivalent to  the quasi-norms $B_2(f)$.
\end{itemize} 
\end{theorem}

\begin{proof}
As mentioned above, for $\Omega$ as in \eqref{def[Bab]}, we have
$1/\Omega (2^{-x}) \ = \ 2^{\alpha |x|_1 + \beta |x|_\infty}, \  x \in \RRdp$. By Theorem 
\ref{Theorem:Representation}, the assertion (i) of the theorem is proven if  the conditions 
\eqref{conditionI}--\eqref{conditionIII}  and  \eqref{condition(a)}--\eqref{condition(b)} with some $\mu > 1/p$ and $\rho < r$, are verified. The condition  \eqref{conditionI} is obvious. 
Put
$
\phi(x): =  \log_2 \{1/\Omega (2^{-x})\}
\ = \ 
\alpha |x|_1 + \beta |x|_\infty,
\ \  x \in \RRdp. 
$
Then the conditions \eqref{conditionII}--\eqref{conditionIII}  and  
\eqref{condition(a)}--\eqref{condition(b)} are equivalent to the following conditions for the function $\phi$, 
\begin{equation} \label{conditionII'}
\phi (x)
\ \le \ 
 \phi (x')+ \log_2 C, \ \ x \le x', \ x, x' \in \RRdp; 
\end{equation}
for every $b \le \log_2 \gamma := (\log_2 \gamma_1,..., \log_2 \gamma_d)$, 
\begin{equation} \label{conditionIII'}
\phi (x + b)
\ \le \ 
\phi (x) + \log_2 C', \ \ x,  x + b \in \RRdp;   
\end{equation}
\begin{equation} \label{condition(a")}
\phi (x) - \mu |x|_1
\ \le \ 
 \phi (x') - \mu |x'|_1 + \log_2 C_1, \ \ x \le x', \ x, x' \in \RRdp; 
\end{equation} 
\begin{equation} \label{condition(b")}
\phi (x) - \rho|x|_1
\ \ge \ 
\phi (x') - \rho|x'|_1 + \log_2 C_2, \ \ x \le x', \ x, x' \in \RRdp.   
\end{equation}

We first consider the case $\beta \ge 0$. Take $\mu$ and $\rho$ with the conditions 
 $1/p < \mu < \alpha$ and $\rho = \alpha + \beta$. The conditions 
\eqref{conditionII'}--\eqref{conditionIII'}  can be easily verified. From the inequality $\alpha -\mu > 0$ and the equation
\begin{equation} \label{eq[phi-mu]}
\phi (x) - \mu |x|_1
\ = \ 
(\alpha -\mu) |x|_1 + \beta |x|_\infty, \ x \in \RRdp. 
\end{equation}
follows  \eqref{condition(a")}. We have  
\begin{equation} \nonumber
\phi (x) - \rho |x|_1
\ = \ 
\beta (|x|_\infty - |x|_1)
\ = \ 
- \beta \min_{j \in [d]} \ \sum_{i \not=j} x_i, \ x \in \RRdp. 
\end{equation}
Hence, we deduce \eqref{condition(b")}.

Let us next consider the case $\beta < 0$. The condition \eqref{conditionIII'} is obvious.  
Take $\mu$ and $\rho$ with the conditions $1/p < \mu < \alpha + \beta$ and $\alpha < \rho < r$. Let 
$x \le x', \ x, x' \in \RRdp$. Assume that $|x|_\infty = x_j$ and $|x'|_\infty = x'_{j'}$. By using the 
inequalities $x_{j'} \le x_j \le x'_j \le x'_{j'}$ and $\alpha - \mu > \alpha + \beta - \mu > 0$ from 
\eqref{eq[phi-mu]} we get 
\begin{equation} \nonumber
\begin{aligned} 
\phi (x) - \mu |x|_1
\ &= \ 
(\alpha -\mu) \sum_{i \not= j,j'} x_i + (\alpha + \beta -\mu) x_j + (\alpha - \mu) x_{j'} \\ 
\ &\le \ 
(\alpha -\mu) \sum_{i \not= j,j'} x'_i + (\alpha + \beta -\mu) x'_{j'} + (\alpha - \mu) x'_j \\ 
\ &= \ 
\phi (x') - \mu |x'|_1.
\end{aligned} 
\end{equation}
The inequality \eqref{condition(a")} is proven. The inequality \eqref{conditionII'} and \eqref{condition(b")} can be proven analogously. Instead the inequalities $\alpha - \mu > \alpha + \beta - \mu > 0$,  in the proof we should use $\alpha > \alpha + \beta > 0$ and $\alpha + \beta - \rho < \alpha  - \rho < 0$, respectively. Thus, the assertion (i) 
is proven.

The assertion (ii) can be proven in a similar way. The assertion (iii) follows from the assertions (i) and (iii).
\end{proof}

\begin{lemma} \label{Lemma[quasi-norm<]}  
Let $0 < p, \theta \le \infty$, $0 < \beta < \min (r, r - 1 + 1/p)$ and $i \in [d]$. 
Then for every function $g \in L_p(\IId)$ represented by  the  series \eqref{representationby[Sigma_{r,i}^d(k)_p]},
there holds the inequality 
\begin{equation*} 
\biggl(\sum_{k \in \ZZ_+}\left\{2^{\beta k} \omega_{r,i}(g,2^{-k})_p\right\}^\theta \biggl)^{1/\theta} 
\ \ll \ 
\left( \sum_{k\in \ZZ_+}
 \left\{2^{\beta k} \| g_k\|_p \right\}^\theta \right)^{1/\theta}
\end{equation*}
with the usual change to sup when $\theta = \infty$, whenever the sum on the right-hand side is finite.
\end{lemma}

\begin{proof}
This lemma can be derived from Lemma \ref{Lemma:IneqMixedomega_{r,i}} by applying Lemma \ref{lemma[HardyIneq]} to the sequences
$b_k :=\omega_{r,i}(g,2^{-k})_p$ and $a_k:= \| g_k\|_p$, $k \in \ZZ_+$, with  
$d=1$, $\psi(x):= 2^{\beta x}$, $\tau := \min(p,1)$, $\delta:= \min (r, r - 1 + 1/p)$.
\end{proof}

\noindent
{\bf Remark} \  Theorem \ref{Theorem:Representation[Ba]} for $a = \alpha {\bf 1}$ and Theorem 
\ref{Theorem:Representation[Bab]} for $\beta = 0$ coincide. This particular case has been proven in \cite{Di11}. Some modifications of the space $\BO$ were introduced in \cite{Di86} where approximations by trigonometric polynomials with frequencies from hyperbolic cross, $n$-widths were investigated from functions from these spaces.

There are many examples of function $\Omega$ for the space $\BO$ for which Theorem \ref{Theorem:Representation} is true with some light natural restrictions and to which it is interesting to extend the results of the present paper. Let us give some important ones of them.
 For  bounded sets $A,B$ with $B \subset A \subset \RRdp$, 
\begin{equation} \label{Omega-AB}
\Omega (t)= \inf_{x \in A}  \prod_{i=1}^d t_i^{x_i} \, \sup_{y \in B}  \prod_{i=1}^d t_i^{y_i}.
\end{equation}
 For univariate functions $\Omega_j, \ j \in [d]$, satisfying the conditions 
\eqref{conditionI}--\eqref{conditionIII},
\[
\Omega(t)= \prod_{j \in [d]}\Omega_j(t_j).
\]
 For $a \in \RRdp$ and a univariate function $\Omega^*$ satisfying the conditions 
\eqref{conditionI}--\eqref{conditionIII},
\[
\Omega(t)= \Omega^* ( \prod_{i=1}^d t_i^{a_i}).
\]
An important case of \eqref{Omega-AB} is the case $\Omega (t):= \inf_{x \in A}  \prod_{i=1}^d t_i^{x_i}$ where 
$A \subset \RRdp$ is a finite set. In this case,  
\[
\BO \ = \ \bigcap_{a \in A} \Ba 
\]
which is a Besov space having several different nonuniform mixed smoothnesses $a \in A$, see \cite{Di84, Di86}.

\section{Sampling recovery} 
\label{Sampling recovery}

\medskip
Let $\Delta \subset \ZZdp$ be given. Put $K(\Delta) :=  \{(k,s): k \in \Delta,\ s \in I^d(k)\}$ and  
denote by $M_r^d(\Delta)$ the set of B-spines 
$M^{(r)}_{k,s}$, $k \in \Delta$, $s \in J_r^d(k)$. 
We define the operator $R_\Delta$ for functions $f$ on $\IId$ by 
\begin{equation} \nonumber
R_\Delta(f) 
:= \ 
\sum_{k \in \Delta} q_k(f)
\ = \ 
\sum_{k \in \Delta} \sum_{s \in J_r^d(k)} c^{(r)}_{k,s}(f)M^{(r)}_{k,s}, 
\end{equation}
and the grid $G(\Delta)$ of points in $\IId$ by
\begin{equation} \nonumber
G(\Delta)
 := \
\{ 2^{-k}s: (k,s) \in K(\Delta)\}. 
\end{equation} 

\begin{lemma} \label{Lemma:R_n=L_n}
The operator $R_\Delta$ defines a linear sampling algorithm 
of the form \eqref{def[L_n]} on the grid $G(\Delta)$. More precisely, 
\begin{equation} \nonumber
R_\Delta(f) 
\ = \ 
L_n(X_n,\Phi_n, f) 
\ = \ 
\sum_{(k,s) \in K(\Delta)}  f(2^{-k}j) \psi_{k,s}, 
\end{equation} 
where $X_n := G(\Delta)= \{2^{-k}s\}_{(k,s) \in K(\Delta)}$, $\Phi_n:= \{ \psi_{k,j}\}_{(k,s) \in K(\Delta)}$,
\begin{equation} \nonumber
n := 
|G(\Delta)|
\ = \
\sum_{k \in \Delta} \prod_{j=1}^d (2^{k_j} + 1),
\end{equation} 
and $\psi_{k,j}$ are explicitly constructed as linear combinations of at most $N$  B-splines 
$M^{(r)}_{k,s} \in  M_r^d(\Delta)$ for some $N \in \NN$ which is independent of $k,j,\Delta$ and $f$.
\end{lemma}

\begin{proof}
This lemma can be proven in a way similar to the proof of \cite[Lemma 3.1]{Di11}.
\end{proof}

Let $\psi: \ZZdp \to \RR_+$. Denote by $\Bpsi$ the space of all functions $f$ on $\IId$  for which the following quasi-norm is finite 
\begin{equation*} 
\|f\|_{\Bpsi}
:= \ 
\biggl( \sum_{k \in {\ZZ}^d_+} \ \{2^{\psi(k)}\|q_k(f)\|_p\}^\theta\biggl)^{1/\theta}. 
\end{equation*}

\begin{lemma} \label{Lemma[|f - R_Delta(f)|_q<]}
Let $0 < p, \theta, q \le \infty$ and  $\psi: \ZZdp \to \RR_+$.  Then for every  $f \in \Bpsi$, we have the following. 
\begin{itemize}
\item[{\rm (i)}] For $p \ge q$,
\begin{equation*} 
\|f - R_\Delta(f)\|_q 
\ \ll \ 
\|f\|_{\Bpsi}
\begin{cases}
\displaystyle
\sup_{k \in \ZZdp \setminus \Delta} 2^{-\psi(k)}, \ & \theta \le \min(p,1), \\
\left(\sum_{k \in \ZZdp \setminus \Delta} \{2^{-\psi(k)}\}^{\theta^*}\right)^{1/\theta^*}, \ & \theta > \min(p,1),
\end{cases}
\end{equation*}
where $\theta^*:= (1/ \min(p,1) - 1/\theta)^{-1}$.
\item[{\rm (ii)}] For $p < q <\infty$, 
\begin{equation*} 
\|f - R_\Delta(f)\|_q 
\ \ll \ 
\|f\|_{\Bpsi}
\begin{cases}
\displaystyle
\sup_{k \in \ZZdp \setminus \Delta} 2^{-\psi(k) + (1/p - 1/q) |k|_1}, \ & \theta \le q, \\
\left(\sum_{k \in \ZZdp \setminus \Delta} \{2^{-\psi(k) + (1/p - 1/q) |k|_1} \}^{q^*}\right)^{1/q^*}, \ & \theta > q,
\end{cases}
\end{equation*}
where $q^*:= (1/q - 1/\theta)^{-1}$.
\item[{\rm (iii)}] For $p < q = \infty$,
\begin{equation*}
\|f - R_\Delta(f)\|_\infty 
\ \ll \ 
\|f\|_{\Bpsi}
\begin{cases}
\displaystyle
\sup_{k \in \ZZdp \setminus \Delta} 2^{-\psi(k) + |k|_1/p}, \ & \theta \le 1 \\
\left(\sum_{k \in \ZZdp \setminus \Delta} \{2^{-\psi(k) + |k|_1/p} \}^{\theta'}\right)^{1/\theta'}, \ & \theta > 1,
\end{cases}
\end{equation*}
where $\theta':= (1 - 1/\theta)^{-1}$.
\end{itemize} 
\end{lemma}

\begin{proof} \\
{\it Case} (i): $p \ge q$. Since $\| \cdot\|_q \le \|\cdot\|_p$, it is sufficient to prove this case for $q=p$.
For an arbitrary $f \in \Bpsi$, by the representation \eqref{eq:B-splineRepresentation}
we have
\begin{equation*}
\|f - R_\Delta(f)\|_p^{\tau} 
\  \ll \ 
\sum_{k \in \ZZdp \setminus \Delta} \| q_k(f)\|_p^{\tau}
\end{equation*}
with any $\tau \le \min(p,1)$. Therefore, if $\theta \le \min(p,1)$, then
by Theorem \ref{Theorem:Representation[Bab]} we get 
\begin{equation} \label{error[f - R_Delta(f)(1)]}
\begin{aligned}
\|f - R_\Delta(f)\|_p  
\  & \ll \ 
 \biggl(\sum_{k \in \ZZdp \setminus \Delta} \| q_k(f)\|_p^{\theta} \biggl)^{1/\theta} \\
\  & \le \ 
\sup_{k \in \ZZdp \setminus \Delta} 2^{-\psi(k)}
\biggl(\sum_{k \in \ZZdp \setminus \Delta} 
\{2^{\psi(k)} \| q_k(f)\|_p\}^{\theta} \biggl)^{1/\theta} \\
 \  & \le \ 
\|f\|_{\Bpsi} \sup_{k \in \ZZdp \setminus \Delta} 2^{-\psi(k)}. 
\end{aligned}
\end{equation}
If $\theta > \min(p,1)$, then 
\begin{equation} \nonumber
\|f - R_\Delta(f)\|_p^\nu
\  \ll \ 
\sum_{k \in \ZZdp \setminus \Delta }\| q_k(f)\|_p^\nu
 \  = \ 
 \sum_{k \in \ZZdp \setminus \Delta } 
\{ 2^{\psi(k)} \| q_k(f)\|_p\}^\nu 
\{2^{- \psi(k)}\}^\nu,
\end{equation}
where $\nu = \min(p,1)$.  Since $\nu/\theta + \nu/\theta^* = 1$, by H\"older's inequality with exponents 
$\theta/\nu, \theta^*/\nu$ and 
Theorem \ref{Theorem:Representation[Bab]} we obtain
\begin{equation} \label{error[f - R_Delta(f)(2)]}
\begin{aligned}
\|f - R_\Delta(f)\|_p
\ & \ll \ 
\biggl( \sum_{k \in \ZZdp \setminus \Delta} 
\{2^{\psi(k)} \| q_k(f)\|_p\}^\theta\biggl)^{1/\theta} 
\biggl(\sum_{k \in \ZZdp \setminus \Delta} 
\{2^{-\psi(k)}\}^{\theta^*}\biggl)^{1/\theta^*} \\
\ & \ll \ 
\|f\|_{\Bpsi} 
\biggl(\sum_{k \in \ZZdp \setminus \Delta} \{2^{-\psi(k)}\}^{\theta^*}\biggl)^{1/\theta^*}. 
\end{aligned}
\end{equation}
This and \eqref{error[f - R_Delta(f)(1)]} prove Case (i).

\noindent
{\it Case} (ii): $p < q < \infty$.
For an arbitrary $f \in \Bpsi$, by the representation \eqref{eq:B-splineRepresentation} 
and \cite[Lemma 5.3]{Di11} have
\begin{equation*} 
\|f - R_\Delta(f)\|_q^q 
\  \ll \ 
 \sum_{k \in \ZZdp \setminus \Delta} \{2^{(1/p - 1/q) |k|_1} \| q_k(f)\|_p\}^q.
\end{equation*}
 Therefore, if $\theta \le q$, then
\begin{equation*}
\begin{aligned} 
\|f - R_\Delta(f)\|_q 
\ & \ll \ 
 \biggl(\sum_{k \in \ZZdp \setminus \Delta}\{2^{(1/p - 1/q) |k|_1} \| q_k(f)\|_p\}^{\theta} \biggl)^{1/\theta} \\
 \ & \ll \ 
\sup_{k \in \ZZdp \setminus \Delta} 2^{-\psi(k) + (1/p - 1/q) |k|_1} 
\biggl(\sum_{k \in \ZZdp \setminus \Delta}\{2^{\psi(k)} \| q_k(f)\|_p\}^{\theta} \biggl)^{1/\theta} \\
 \ &\ll \ 
\|f\|_{\Bpsi}
\sup_{k \in \ZZdp \setminus \Delta} 2^{-\psi(k) + (1/p - 1/q) |k|_1}.
 \end{aligned}
\end{equation*}
If $\theta > q$, then
\begin{equation*}
\begin{aligned} 
\|f - R_\Delta(f)\|_q^q 
\ & \ll \ 
\sum_{k \in \ZZdp \setminus \Delta }\{2^{(1/p - 1/q) |k|_1} \| q_k(f)\|_p\}^q \\
 \ & = \ 
 \sum_{k \in \ZZdp \setminus \Delta }\{2^{\psi(k)}\| q_k(f)\|_p\}^q\{2^{- \psi(k) +(1/p - 1/q)|k|_1}\}^q.
 \end{aligned}
\end{equation*}
Hence, similarly to \eqref{error[f - R_Delta(f)(2)]}, we get
\begin{equation*}
\|f - R_\Delta(f)\|_q 
\ \ll \ 
\|f\|_{\Bpsi} 
\biggl(\sum_{k \in \ZZdp \setminus \Delta} \{2^{-\psi(k) + (1/p - 1/q) |k|_1}\}^{q^*}\biggl)^{1/q^*}.
 \end{equation*}
 This completes the proof of Case (ii).

\noindent
{\it Case} (iii): $p < q = \infty$. Case (iii) can be proven analogously to Case (ii) by using the inequality
\begin{equation*} 
\|f - R_\Delta(f)\|_\infty 
\  \ll \ 
 \sum_{k \in \ZZdp \setminus \Delta} 2^{|k|_1/p} \| q_k(f)\|_p.
\end{equation*}
\end{proof}

Denote by $\RRbp$ the set of triples $(p,\theta,q)$ such that $0 < p,\theta,q  \le \infty$.
According to Lemma \ref{Lemma[|f - R_Delta(f)|_q<]}, depending on the relationship between 
$p,\theta,q$ for $(p,\theta,q) \in\RRbp$, the error $\|f - R_\Delta(f)\|_q$ of the approximation of $f \in \Bpsi$ has  an upper bound of two different forms: either
\begin{equation} \label{ineq[sup-error]}
\|f - R_\Delta(f)\|_q 
\ \ll \ 
\|f\|_{\Bpsi}
\sup_{k \in \ZZdp \setminus \Delta} 2^{-\psi(k) + (1/p - 1/q)_+ |k|_1},
\end{equation}
or for some $0 < \tau < \infty$,
\begin{equation} \label{ineq[sum-error]}
\|f - R_\Delta(f)\|_q 
\ \ll \ 
\|f\|_{\Bpsi}
\biggl(\sum_{k \in \ZZdp \setminus \Delta} \{2^{-\psi(k)+ (1/p - 1/q)_+ |k|_1}\}^\tau\biggl)^{1/\tau}.
\end{equation}

Let us decompose $\RRbp$ into two sets $A$ and $B$ with $A \cap B = \varnothing$ as follows. A triple 
$(p,\theta,q) \in\RRbp$ belongs to $A$ if and only if for $(p,\theta,q)$ there holds \eqref{ineq[sup-error]}, and belongs to $B$ if and only if for $(p,\theta,q)$ there holds \eqref{ineq[sum-error]}. 
By Lemma \ref{Lemma[|f - R_Delta(f)|_q<]}, we can see that
\[
A = \{p \ge q, \ \theta \le \min(p,1)\}\cup \{p < q, \ \theta \le q\}\cup \{p < q = \infty, \ \theta \le 1\},
\] 
and 
\[
B = \{p \ge q, \ \theta > \min(p,1)\}\cup \{p < q, \ \theta > q\}\cup \{p < q = \infty, \ \theta > 1\}.
\] 

We construct special sets $\Delta(\xi)$ parametrized by $\xi > 0$, for the recovery of functions $f \in \Uab$ by $R_{\Delta(\xi)}(f)$.  Let $0 < p, \theta, q \le \infty$ and $\alpha \in \RR_+$, $\beta \in \RR$ be given.
We fix a number $\varepsilon$ so that
\begin{equation*} 
0 < \varepsilon < \min (\alpha - (1/p - 1/q)_+, |\beta|),
\end{equation*} 
 and define the set $\Delta (\xi)$ for $\xi > 0$ by
\begin{equation*} 
\Delta (\xi)
:= \
\begin{cases}
\{k \in \ZZdp: (\alpha - (1/p - 1/q)_+)|k|_1 + \beta |k|_\infty \le \xi\}, \ & (p,\theta,q) \in A, \\
\{k \in \ZZdp: (\alpha - (1/p - 1/q)_+ + \varepsilon/d)|k|_1 + (\beta - \varepsilon) |k|_\infty\le \xi\},
 \ & (p,\theta,q) \in B, \ \beta > 0, \\
\{k \in \ZZdp: (\alpha - (1/p - 1/q)_+ - \varepsilon)|k|_1 + (\beta + \varepsilon) |k|_\infty\le \xi\},
 \ & (p,\theta,q) \in B, \ \beta < 0.
\end{cases}
\end{equation*}

Preliminarily note that for $(p,\theta,q) \in A$, $\Delta(\xi)$ is defined as the set
$\{k \in \ZZdp: (\alpha - (1/p - 1/q)_+)|k|_1 + \beta |k|_\infty \le \xi\}$, but for $(p,\theta,q) \in B$,  $\Delta(\xi)$ is defined as an extension of the last one parametrized by $\varepsilon$. We will give a detailed comment on this substantial difference in a remark at the end of the next section where the optimality and sparsity are investigated.

\begin{theorem} \label{Theorem[UpperBoundUab]}
Let $\ 0 < p, \theta, q \le \infty$ and $\alpha \in \RR_+$, $\beta \in \RR$, $\beta \not=0$, such that
\begin{equation} \nonumber
1/p < \min (\alpha,\alpha + \beta) \le  \max(\alpha,\alpha + \beta) < r. 
\end{equation}
Then we have the following upper bound
\begin{equation}  \label{Ineq[UpperBoundUab]}
\sup_{f \in \Uab} \ \|f - R_{\Delta(\xi)}(f)\|_q 
\ \ll \ 
2^{-\xi}.
\end{equation}
\end{theorem}

\begin{proof}
If $(p,\theta,q) \in A$, by Lemma \ref{Lemma[|f - R_Delta(f)|_q<]}, we have
\begin{equation} \nonumber 
\sup_{f \in \Uab} \ \|f - R_{\Delta(\xi)}(f)\|_q 
\ \ll \ 
\sup_{k \in \ZZdp \setminus \Delta(\xi)} 2^{- (\alpha - (1/p - 1/q)_+)|k|_1 + \beta |k|_\infty )}
\ \ll \
2^{-\xi}.
\end{equation}

We next consider the case $(p,\theta,q) \in B$. In this case, by Lemma \ref{Lemma[|f - R_Delta(f)|_q<]}, we have
\begin{equation} \nonumber
\sup_{f \in \Uab} \ \|f - R_{\Delta(\xi)}\|_q^\tau 
\ \ll \ 
\sum_{k \in \ZZdp \setminus \Delta(\xi)} 2^{- \tau(\alpha - (1/p - 1/q)_+)|k|_1 + \beta |k|_\infty )}
\end{equation}
for $\tau = \theta^*, q^*, \theta'$. For simplicity we prove the case $(p,\theta,q) \in B$ for $\tau = 1$ and
$p \ge q$, the general case can be proven similarly. In this particular case, we get 
\begin{equation} \label{Ineq[UpperBoundUab(2)]}
\sup_{f \in \Uab} \ \|f - R_{\Delta(\xi)}\|_q 
\ \ll \ 
\sum_{k \in \ZZdp \setminus \Delta(\xi)} 2^{- \alpha |k|_1 - \beta |k|_\infty}
\ =: \Sigma(\xi).
\end{equation}

We first assume that  $\beta > 0$. It is easy to verify that for every $\xi > 0$,
\begin{equation} \label{asymp[Sigma(1)]}
\Sigma(\xi)
\ \asymp \
\int_{W(\xi)} 2^{-(\alpha {\bf 1},x) - \beta M(x)} dx,
\end{equation}
where $M(x):= \max_{j \in [d]} x_j$ for $x \in \RRd$, and 
\[
W(\xi):= \{x \in \RRdp: (\alpha + \varepsilon/d) ({\bf 1},x) + (\beta - \varepsilon) M(x) > \xi\}.
\]
Putting
\[
V(\xi,s):= \{x \in W(\xi): \xi + s - 1 \le \alpha  ({\bf 1},x) + \beta M(x) < \xi + s \}, \ s \in \NN,
\]
from \eqref{asymp[Sigma(1)]} we have
\begin{equation} \label{asymp[Sigma(2)]}
\Sigma(\xi)
\ \asymp \
2^{-\xi} \sum_{s=1}^\infty 2^{-s}|V(\xi, s)|.
\end{equation}

Let us estimate $|V(\xi,s)|$. Put $V^*(\xi,s):= V(\xi, s) - x^*$, where $x^*:= (\nu d)^{-1}\xi {\bf 1}$ and 
$\nu:= \alpha + \beta/d$. For every 
$y = x - x^* \in V^*(\xi,s)$, from the equation $({\bf 1},x^*) = \xi/\nu$ and the inequality 
$\alpha({\bf 1},x) + \beta M(x) < \xi + s $ we get
\begin{equation} \label{def[V'(1)]}
\alpha({\bf 1},y) + \beta M(y) \ < \ s.
\end{equation}
On the other hand, for every $x \in V(\xi,s)$, from  the inequality 
$\alpha({\bf 1},x) + \beta M(x) < \xi + s$ and 
$(\alpha + \varepsilon/d)({\bf 1},x) + (\beta - \varepsilon) M(x) > \xi$ we get
$ M(x) - ({\bf 1},x)/d < \varepsilon^{-1} s$. This inequality together with the inequality 
$\alpha({\bf 1},x) + \beta M(x) \ge   s - 1$ gives 
$({\bf 1},x) \ge \xi/\nu  + ((1 - \varepsilon^{-1}\beta)s + 1)/\nu$ for every $x \in V(\xi,s)$.
Hence, for every $y = x - x^* \in V^*(\xi,s)$,
\begin{equation} \label{def[V'(2)]}
({\bf 1},y) \ \ge \ ((1 - \varepsilon^{-1}\beta)s + 1)/\nu.
\end{equation}
This means that $V^*(\xi,s) \subset V'(s)$ for every $\xi > 0$, where $V'(s) \subset \RRd$ is the set of 
all $y \in \RRd$ given by the conditions \eqref{def[V'(1)]} and \eqref{def[V'(2)]}. Since $V'(s)$ is a bounded polyhedron and consequently, 
\begin{equation} \nonumber
|V(\xi,s)| \ = \ |V^*(\xi,s)| \ \le \ |V'(s)| \ \asymp \ s^d,
\end{equation} 
combining \eqref{Ineq[UpperBoundUab(2)]} and \eqref{asymp[Sigma(2)]}, we obtain
\begin{equation} \nonumber
\sup_{f \in \Uab} \ \|f - R_{\Delta(\xi)}\|_q 
\ \ll \ 
2^{-\xi} \sum_{s=1}^\infty 2^{-s} s^d
\ \asymp  \ 
2^{-\xi}.
\end{equation}

If $\beta < 0$, similarly to \eqref{Ineq[UpperBoundUab(2)]} and \eqref{asymp[Sigma(1)]}, we have
for every $\xi > 0$,
\begin{equation} \nonumber
\Sigma(\xi)
\ \asymp \
\int_{W'(\xi)} 2^{-(\alpha {\bf 1},x) - \beta M(x)} dx,
\end{equation} 
where
\[
W'(\xi):= \{x \in \RRdp: (\alpha - \varepsilon) ({\bf 1},x) + (\beta + \varepsilon) M(x) > \xi\}.
\]

From the last relation, similarly to the proof for the case $\beta > 0$,
we prove \eqref{Ineq[UpperBoundUab]} for the case $\beta < 0$.
\end{proof}

We construct special sets $\Delta'(\xi)$ parametrized by $\xi > 0$, for the recovery of functions $f \in \Ua$ by $R_{\Delta'(\xi)}(f)$. Let $0 < p,q, \theta \le \infty$ and $a \in \RRdp$ be given. In what follows, we assume the following restriction on the smoothness $a$ of $\Ba$:
\begin{equation} \label{condition[a]}
1/p < a_1 < a_2 \le ... \le a_d < r. 
\end{equation} 
We fix a number $\varepsilon$ so that
\begin{equation*} 
0 < \varepsilon < a_2  - a_1,
\end{equation*} 
 and define the set $\Delta'(\xi)$ for $\xi > 0$, by

 \begin{equation*} 
\Delta'(\xi)
:= \
\begin{cases}
\{k \in \ZZdp: (a,k) - (1/p - 1/q)_+ |k|_1 \le \xi\}, \ & (p,\theta,q) \in A, \\
\{k \in \ZZdp: (a(\varepsilon),k) - (1/p - 1/q)_+ |k|_1 \le \xi\}, \ & (p,\theta,q) \in B,
\end{cases}
\end{equation*}
 where $a(\varepsilon) = (a_1, a_2 - \varepsilon,...,a_d - \varepsilon)$.

\begin{theorem} \nonumber
Let $\ 0 < p, \theta, q \le \infty$ and $a \in \RRdp$ satisfying the condition \eqref{condition[a]} and
\begin{equation} \nonumber
1/p < a_1 < a_d < r. 
\end{equation}
Then we have the following upper bound
\begin{equation}  \label{Ineq[UpperBoundUa]}
\sup_{f \in \Ua} \ \|f - R_{\Delta'(\xi)}(f)\|_q 
\ \ll \ 
2^{-\xi}.
\end{equation}
\end{theorem}

\begin{proof}
Let us first consider the case $(p,\theta,q) \in A$. In this case, by Lemma \ref{Lemma[|f - R_Delta(f)|_q<]}, we have
\begin{equation} \nonumber 
\sup_{f \in \Ua} \ \|f - R_{\Delta'(\xi)}(f)\|_q 
\ \ll \ 
\sup_{k \in \ZZdp \setminus \Delta'(\xi)} 2^{- ((a,k) - (1/p - 1/q)_+|k|_1)}
\ \ll \
2^{-\xi}.
\end{equation}

We next treat the case $(p,\theta,q) \in B$. In this case, by Lemma \ref{Lemma[|f - R_Delta(f)|_q<]}, we have
\begin{equation} \nonumber
\sup_{f \in \Ua} \ \|f - R_{\Delta'(\xi)}(f)\|_q^\tau 
\ \ll \ 
\sum_{k \in \ZZdp \setminus \Delta'(\xi)} 2^{- \tau((a,k) - (1/p - 1/q)_+|k|_1)},
\end{equation}
for $\tau = \theta^*, q^*, \theta'$. For simplicity we prove the case $(p,\theta,q) \in B$ for $\tau = 1$ and
$(1/p - 1/q)_+) = 0$, the general case can be proven similarly. In this particular case, we get 
\begin{equation} \label{Ineq[UpperBoundUa(2)]}
\sup_{f \in \Ua} \ \|f - R_{\Delta'(\xi)}(f)\|_q 
\ \ll \ 
\sum_{k \in \ZZdp \setminus \Delta'(\xi)} 2^{- (a,k)} =: \Sigma(\xi).
\end{equation}
It is easy to verify that for every $\xi > 0$,
\begin{equation} \label{asymp[Sigma(1a)]}
\Sigma(\xi)
\ \asymp \
\int_{W(\xi)} 2^{-(a,x)} dx,
\end{equation}
where
\begin{equation} \nonumber
W(\xi):= \{x \in \RRdp: (a',x) > \xi\}.
\end{equation}
We put 
\begin{equation} \nonumber
V(\xi,s):= \{x \in W(\xi): \xi + s - 1 \le (a,x) < \xi + s \}, \ s \in \NN.
\end{equation}
then from \eqref{asymp[Sigma(1a)]} we have
\begin{equation} \label{asymp[Sigma(2a)]}
\Sigma(\xi)
\ \asymp \
2^{-\xi} \sum_{s=1}^\infty 2^{-s}|V(\xi,s)|.
\end{equation}

Let us estimate $|V(\xi,s)|$. Put $V^*(\xi,s):= V(\xi, s) - x^*$, where $x^*:= (a_1)^{-1}\xi e^1$. For every 
$y = x - x^* \in V^*(\xi,s)$, from the equation $(a,x^*) = \xi$ and the inequality 
$(a,x) < \xi + s $ we get $(a,y) < s$ and therefore,
\begin{equation} \label{def[V'(1a)]}
y_j \ < \ s/a_j, \ j \in [d].
\end{equation}
On the other hand, for every $x \in V(\xi,s)$, from  the inequality 
$(a,x) < \xi + s$ and $(a,x) - \varepsilon({\bf 1}',x) = (a',x) > \xi$ we get
$ ({\bf 1}',x) < \varepsilon^{-1} s$, where ${\bf 1}' := (0,1,1,...,1) \in \RRd$. This inequality together with the inequality $a_1x_1 + a_d({\bf 1}',x) \ge (a,x) \ge \xi + s - 1$ gives 
$x_1 \ge \xi/a_1  + ((1 - \varepsilon^{-1}a_d)s + 1)/a_1$ for every $x \in V(\xi,s)$.
Hence, for every $y = x - x^* \in V^*(\xi,s)$,
\begin{equation} \label{def[V'(2a)]}
y_1 \ \ge \ ((1 - \varepsilon^{-1}a_d)s + 1)/a_1, \ \ y_j \ \ge \ 0, \ j =2,...,d.
\end{equation}
This means that $V^*(\xi,s) \subset V'(s)$ for every $\xi > 0$, where $V'(s) \subset \RRd$ is the box of 
all $y \in \RRd$ given by the conditions \eqref{def[V'(1a)]} and \eqref{def[V'(2a)]}. Since 
\begin{equation} \nonumber
|V(\xi,s)| \ = \ |V^*(\xi,s)| \ \le \ |V'(s)| \ \asymp \ s^d,
\end{equation}
by \eqref{Ineq[UpperBoundUa(2)]} and \eqref{asymp[Sigma(2a)]}, we obtain
\begin{equation} \nonumber
\sup_{f \in \Ua} \ \|f - R_{\Delta(\xi)}\|_q 
\ \ll \ 
2^{-\xi} \sum_{s=1}^\infty 2^{-s} s^d
\ \asymp  \ 
2^{-\xi}.
\end{equation}
\end{proof}

\medskip
\noindent
{\bf Remark} \ The grids $G(\Delta(\xi))$ and $G(\Delta'(\xi))$ defined for 
$(p,\theta,q) \in A$ or $(p,\theta,q) \in B$ with $\beta >0$, were employed in \cite{Di90, Di91, Di92} for sampling recovery of periodic functions from an intersection of spaces of different mixed smoothness.
Independently of our paper, similar upper bounds of the error in $W^\beta_2(\TTd)$-norm of linear sampling algorithms by interpolation for functions from periodic space $W^{\alpha {\bf 1}}_2(\TTd)$ have recently obtained in \cite{GH14}.

\section{Sparsity and optimality} \label{Optimality}

Denote by $\lfloor x \rfloor$ the integer part of $x \in \RR_+$.
\begin{lemma} \label{Lemma[SamplingNumber]}
Let $\ 0 < p, \theta, q \le \infty$ and $\alpha \in \RR_+$, $\beta \in \RR$, $\beta \not=0$, such that
\begin{equation} \nonumber
1/p < \min (\alpha,\alpha + \beta) \le  \max(\alpha,\alpha + \beta) < r. 
\end{equation}
Then we have 
\begin{equation}  \label{asymp[Sampling Number]}
|G(\Delta(\xi))|
\ \asymp \ 
\sum_{k \in \Delta(\xi)} 2^{|k|_1}
\ \asymp \ 
2^{\xi/ \nu},
\end{equation}
where 
\begin{equation} \label{def[nu]}
\nu 
:= \
\begin{cases}
\alpha  + \beta/d - (1/p - 1/q)_+, \ &  \beta > 0, \\
\alpha  + \beta - (1/p - 1/q)_+, \ &  \beta < 0.
\end{cases}
\end{equation}
\end{lemma}

\begin{proof}
The first asymptotic equivalence in \eqref{asymp[Sampling Number]} follows from the definitions. Let us prove the second one. For simplicity we prove it for the case where $p \ge q$, the general case can be proven similarly.

Let us first consider the case $(p,\theta,q) \in B, \beta > 0$. 
It is easy to verify that for every $\xi > 0$,
\begin{equation} \label{asymp[SamplingNumber(1)]}
\sum_{k \in \Delta(\xi)} 2^{|k|_1}
\ \asymp \
\int_{W(\xi)} 2^{({\bf 1},x)} dx,
\end{equation}
where
\begin{equation} \nonumber
W(\xi):= \{x \in \RRdp: (\alpha + \varepsilon/d) ({\bf 1},x) + (\beta - \varepsilon) M(x) \le \xi\}
\end{equation}
and  $M(x):= \max_{j \in [d]} x_j$ for $x \in \RRd$. 
We put 
\begin{equation} \nonumber
V(\xi,s):= \{x \in W(\xi): \xi/\nu + s - 1 \le ({\bf 1},x) < \xi/\nu + s \}, \ s \in \ZZ_+.
\end{equation}
From the inequalities $\beta > \varepsilon$ and  $M(x) - ({\bf 1},x)/d \ge 0, \ x \in \RRdp,$ one can verify that 
for every $x \in W(\xi)$, $({\bf 1},x) \le \xi/\nu$. Hence,  we have
\begin{equation} \label{asymp[SamplingNumber(2)]}
\int_{W(\xi)} 2^{({\bf 1},x)} dx
\ \ll \
2^{\xi/\nu} \sum_{s=0}^{\lceil \xi/\nu \rceil} 2^{-s}|V(\xi, s)|.
\end{equation}

Let us estimate $|V(\xi,s)|$. Put $V^*(\xi,s):= V(\xi, s) - x^*$, where $x^*:= (\nu d)^{-1}\xi {\bf 1}$. From the equation $({\bf 1},x^*) = \xi/\nu$, we get for every $y = x - x^* \in V^*(\xi,s)$, 
\begin{equation} \label{def[SamplingNumberV'(1)]}
s - 1 \ \le \ ({\bf 1},y) \ < \ s.
\end{equation}
and
\begin{equation} \label{def[SamplingNumberV'(2)]}
(\alpha + \varepsilon/d) ({\bf 1},y) + (\beta - \varepsilon) M(y) \ \le \ 0.
\end{equation}
This means that $V^*(\xi,s) \subset V'(s)$ for every $\xi > 0$, where $V'(s) \subset \RRd$ is the set of 
all $y \in \RRd$ given by the conditions \eqref{def[SamplingNumberV'(1)]} and \eqref{def[SamplingNumberV'(2)]}. Notice that $V'(s)$ is a bounded polyhedron and $|V'(s)| \ \asymp \ s^{d-1}$. Hence, by the inequality  
\begin{equation} \nonumber
|V(\xi,s)| \ = \ |V^*(\xi,s)| \ \le \ |V'(s)|,
\end{equation}
 \eqref{asymp[SamplingNumber(1)]} and \eqref{asymp[SamplingNumber(2)]}, we prove the upper bound in 
\eqref{asymp[Sampling Number]}: 
\begin{equation} \nonumber
\sum_{k \in \Delta(\xi)} 2^{|k|_1}
\ \ll \
2^{\xi/\nu} \sum_{s=0}^\infty 2^{-s}s^{d-1}
\ \asymp \
2^{\xi/\nu}. 
\end{equation}
 To prove the lower bound for this case, we take $k^* := \lfloor \xi/d\nu \rfloor {\bf 1} \in \ZZdp$. It is easy to check $k^* \in \Delta(\xi)$ and consequently,
\begin{equation} \nonumber
\sum_{k \in \Delta(\xi)} 2^{|k|_1}
\ \ge \
 2^{|k^*|_1}
 \ \gg \
2^{\xi/\nu}. 
\end{equation}

The case $(p,\theta,q) \in B, \beta < 0$ can be proven similarly with a slight modification. 
To prove the case $(p,\theta,q) \in A$ it is enough to  put $\varepsilon = 0$ in the proof of the case 
$(p,\theta,q) \in B$. 
\end{proof}

\begin{lemma} \label{Lemma[SamplingNumberUa]}
Let $\ 0 < p, \theta, q \le \infty$ and $a \in \RRdp$ satisfying the condition \eqref{condition[a]} and
\begin{equation} \nonumber
1/p < a_1 < a_d < r. 
\end{equation}
Then we have 
\begin{equation}  \label{asymp[Sampling NumberUa]}
|G(\Delta'(\xi))|
\ \asymp \ 
\sum_{k \in \Delta'(\xi)} 2^{|k|_1}
\ \asymp \ 
2^{\xi/(a_1 - (1/p - 1/q)_+)}.
\end{equation}
\end{lemma}

\begin{proof}
The first asymptotic equivalence in  \eqref{asymp[Sampling NumberUa]} follows from the definitions. Let us prove the second one. For simplicity we prove it for the case where $p \ge q$, the general case can be proven similarly.

Let us first consider the case $(p,\theta,q) \in B$. 
It is easy to verify that for every $\xi > 0$,
\begin{equation} \label{asymp[SamplingNumber(1)Ua]}
\sum_{k \in \Delta'(\xi)} 2^{|k|_1}
\ \asymp \
\int_{W(\xi)} 2^{({\bf 1},x)} dx,
\end{equation}
where
\begin{equation} \nonumber
W(\xi):= \{x \in \RRdp:(a',x)  \le \xi\}.
\end{equation}
We put 
\begin{equation} \nonumber
V(\xi,s):= \{x \in W(\xi): \xi/a_1+ s - 1 \le ({\bf 1},x) < \xi/a_1 + s \}, \ s \in \ZZ_+.
\end{equation}
One can verify that 
for every $x \in W(\xi)$, $({\bf 1},x) \le \xi/a_1$. Hence,  we have
\begin{equation} \label{asymp[SamplingNumber(2)Ua]}
\int_{W(\xi)} 2^{({\bf 1},x)} dx
\ \ll \
2^{\xi/a_1} \sum_{s=0}^{\lceil \xi/a_1 \rceil} 2^{-s}|V(\xi, s)|.
\end{equation}

Let us estimate $|V(\xi,s)|$. Put $V^*(\xi,s):= V(\xi, s) - x^*$, where $x^*:= (a_1)^{-1}\xi e^1$. From the equation $({\bf 1},x^*) = \xi/a_1$, we get for every $y = x - x^* \in V^*(\xi,s)$, 
\begin{equation} \label{def[SamplingNumberV'(1)Ua]}
s - 1 \ \le \ ({\bf 1},y) \ < \ s.
\end{equation}
and
\begin{equation} \label{def[SamplingNumberV'(2)Ua]}
(a',y) \ \le \ 0.
\end{equation}
This means that $V^*(\xi,s) \subset V'(s)$ for every $\xi > 0$, where $V'(s) \subset \RRd$ is the set of 
all $y \in \RRd$ given by the conditions \eqref{def[SamplingNumberV'(1)Ua]} and \eqref{def[SamplingNumberV'(2)Ua]}. Notice that $V'(s)$ is a bounded polyhedron and $|V'(s)| \ \asymp \ s^{d-1}$. Hence, by the inequality  
\begin{equation} \nonumber
|V(\xi,s)| \ = \ |V^*(\xi,s)| \ \le \ |V'(s)|,
\end{equation}
 \eqref{asymp[SamplingNumber(1)Ua]} and \eqref{asymp[SamplingNumber(2)Ua]}, we obtain the upper bound in  
\eqref{asymp[Sampling NumberUa]}:
\begin{equation} \nonumber
\sum_{k \in \Delta'(\xi)} 2^{|k|_1}
\ \ll \
2^{\xi/a_1} \sum_{s=0}^\infty 2^{-s}s^{d-1}
\ \asymp \
2^{\xi/a_1}. 
\end{equation}
 To prove the lower bound, we take $k^* := \lfloor \xi/a_1 \rfloor e^1 \in \ZZdp$. It is easy to check 
$k^* \in \Delta'(\xi)$ and consequently,
\begin{equation} \nonumber
\sum_{k \in \Delta'(\xi)} 2^{|k|_1}
\ \ge \
 2^{|k^*|_1}
 \ \gg \
2^{\xi/a_1}. 
\end{equation}
\end{proof}

\medskip
\noindent
{\bf Remark} \ The grids of sample points $G(\Delta(\xi))$ and $G(\Delta'(\xi))$ are sparse and have much less elements than the standard dyadic full grids $G(\Delta_1(\xi))$ and Smolyak grids $G(\Delta_2(\xi))$ which give the same recovery error, where
$\Delta_1(\xi):= \{k \in \ZZdp: \lambda|k|_\infty  \le \xi \}$ and 
$\Delta_2(\xi):= \{k \in \ZZdp: \lambda|k|_1  \le \xi \}$ 
and the number $\lambda:= \nu$ is as in \eqref{def[nu]} for $G(\Delta(\xi))$ and 
$\lambda:= a_1 - (1/p - 1/q)_+$ for $G(\Delta'(\xi))$. For instance, the linear sampling algorithms $R_{\Delta_i(\xi)}$, $i=1,2$, on the grids $G(\Delta_i(\xi))$ gives the worst case error
\begin{equation}  \nonumber
\sup_{f \in \Uab} \ \|f - R_{\Delta_i(\xi)}(f)\|_q 
\ \asymp \ 
2^{-\xi}.
\end{equation}
The number of sample points in $G(\Delta_1(\xi))$ is
$|G(\Delta_1(\xi))| \asymp 2^{d \xi/\nu}$ and in $G(\Delta_2(\xi))$ is
$|G(\Delta_2(\xi))| \asymp 2^{\xi/\nu} \xi^{d-1}$. Whereas, due to Theorem \ref{Theorem[UpperBoundUab]} and 
Lemma \ref{Lemma[SamplingNumber]} we can get the same error by the linear sampling algorithm $R_{\Delta(\xi)}$ on the grids $G(\Delta(\xi))$ with the number of sample points $|G(\Delta(\xi))| \asymp 2^{\xi/\nu}$.

\medskip

The following two theorems show that the linear sampling sampling algorithms $R_{\Delta(\xi)}$ on sparse grids 
$G(\Delta(\xi))$, and  $R_{\Delta'(\xi)}$ on sparse grids $G(\Delta'(\xi))$
are asymptotically optimal in the sense of the quantity $r_n$.

\begin{theorem} \label{Theorem[r_nUab]}
Let $\ 0 < p, \theta, q \le \infty$ and $\alpha \in \RR_+$, $\beta \in \RR$, $\beta \not = 0$, such that
\begin{equation} \nonumber
1/p < \min (\alpha,\alpha + \beta) \le  \max(\alpha,\alpha + \beta) < r. 
\end{equation}
Assume that for a given $n \in \ZZ_+$, $\xi_n$ is the largest nonnegative number such that
\begin{equation} \label{[|G(Delta(xi_n)|<]}
|G(\Delta(\xi_n))| \ \le \ n.
\end{equation}  
Then $R_{\Delta(\xi_n)}$ defines an asymptotically optimal linear sampling algorithm for
$r_n:= r_n(\Uab, L_q) $ by
\begin{equation} \label{[R_Delta(xi_n))]}
R_{\Delta(\xi_n)}(f) 
\ = \ 
L_n(X_n^*,\Phi_n^*,f)
\ = \ 
\sum_{(k,s) \in K(\Delta(\xi_n))}  f(2^{-k}s) \psi_{k,s}, 
\end{equation} 
where $X_n^* := G(\Delta(\xi_n))= \{ 2^{-k}s\}_{(k,s) \in K(\Delta(\xi_n))}$, 
$\Phi_n^*:= \{ \psi_{k,s}\}_{(k,s) \in K(\Delta(\xi_n))}$,
and we have the following asymptotic orders
\begin{equation} \label{asymp[r_nUab]}
\sup_{f \in \Uab}\ \|f - R_{\Delta(\xi_n)}(f) \|_q
\ \asymp \
r_n 
\ \asymp \
\begin{cases}
n^{- \alpha  - \beta/d + (1/p - 1/q)_+}, \ &  \beta > 0, \\
n^{- \alpha  - \beta + (1/p - 1/q)_+}, \ &  \beta < 0.
\end{cases}
\end{equation}
\end{theorem}

\begin{proof} \\
{\em Upper bounds}. 
Due to Lemma \ref{Lemma[SamplingNumber]} we have 
\begin{equation}  \nonumber
n 
\ \asymp \
2^{\xi_n/\nu}
\ \asymp \ 
|G(\Delta(\xi_n))|
\ \le \ 
n,
\end{equation}
where $\nu$ is as in \eqref{def[nu]}. Hence, we find
\begin{equation} \label{asymp[2^xi]}
2^{-\xi_n}
\ \asymp \
\begin{cases}
n^{- \alpha  + \beta/d - (1/p - 1/q)_+}, \ &  \beta > 0, \\
n^{- \alpha  + \beta - (1/p - 1/q)_+}, \ &  \beta < 0.
\end{cases}
\end{equation}
By Lemma \ref{Lemma:R_n=L_n} and \eqref{[|G(Delta(xi_n)|<]}, $R_{\Delta(\xi_n)}$ is a linear sampling algorithm of the form \eqref{def[L_n]} as in \eqref{[R_Delta(xi_n))]} and consequently, from  Theorem \ref{Ineq[UpperBoundUab]} we get
\begin{equation} \nonumber
r_n 
\ \le \
\sup_{f \in \Uab} \ \|f - R_{\Delta(\xi_n)}(f)\|_q 
\ \ll \ 
2^{-\xi_n}.
\end{equation}
These relations together with \eqref{asymp[2^xi]} proves the upper bounds of \eqref{asymp[r_nUab]}.

\noindent
{\em Lower bounds}. We need the following auxiliary result. If $W \subset L_q$, then we have
\begin{equation} \label{[r_n(W, L_q)>]}
r_n(W, L_q(\IId)) 
\ \gg \
\inf_{X_n=\{x^j\}_{j=1}^m \subset \IId} \ \ \sup_{f \in W: \ f(x^j) = 0, \ j=1,...,n} \ \|f\|_q. 
\end{equation}
For the proof of this inequality see \cite[Proposition 19]{NoT}. Since $\|f\|_q \ge \|f\|_p$ for $p \ge q$, it is sufficient to prove the lower bound for the case $p \le q$.
Fix a number $r' = 2^m$ with integer $m$ so that $\max(\alpha,\alpha + \beta) <  \min(r', r' - 1 + 1/p)$. 

We first treat the case $\beta > 0$. Put $k^*= k^*(\eta) := \eta{\bf 1}$ for an integer $\eta > m$. Consider the boxes  $J(s) \subset \IId$
\begin{equation} \nonumber
J(s)
:= \
\{x \in \IId: 2^{- \eta + m} s_j \le x_j < 2^{- \eta + m} (s_j +1), \ j \in [d]\}, \ s \in Z(\eta), 
\end{equation}
where
\begin{equation} \nonumber
Z(\eta)
:= \
\{s \in \ZZdp: 0 \le  s_j \le  2^{\eta-m} - 1, \ j \in [d]\}. 
\end{equation}
For a given $n$, we find $\eta$ satisfying the relations 
\begin{equation}  \label{def[eta]}
n 
\ \asymp \
2^{|k^*|_1}
\ \asymp \
2^{d(\eta - m)}
\ = \ 
|Z(\eta)|
\ \ge \ 
2n.
\end{equation}
Let $X_n = \{x^j\}_{j=1}^n$ be an arbitrary subset of $n$ points in $\IId$. Since $J(s) \cap J(s') = \varnothing$ for 
$s \not= s'$, and $|Z(\eta)| \ge 2n$, there is $Z^*(\eta) \subset Z(\eta)$ such that $|Z^*(\eta)| \ge n$ and
\begin{equation} \label{cap[X_n]}
X_n \cap \{\cup_{s \in Z^*(\eta)} J(s)\} \ = \ \varnothing.
\end{equation}
Consider the function $g^* \in \Sigma_r^d(k^*)$ defined by
\begin{equation} \label{[g^*Uab(1)]}
g^*
:= \
\lambda 2^{- \alpha |k^*|_1 - \beta |k^*|_\infty + |k^*|_1/p} \sum_{s \in Z^*(\eta)} M_{k^*,s+r'/2},
\end{equation}
where $M_{k^*,s+r'/2}$ are B-splines of order $r'$.
Since $|Z^*(\eta)| \asymp 2^{|k^*|_1}$, by \eqref{eq:StabIneq} we have
\begin{equation} \label{asymp[|g^*|_q]}
\|g^*\|_q
\ \asymp \
\lambda 2^{- \alpha |k^*|_1 - \beta |k^*|_\infty + (1/p - 1/q)|k^*|_1},
\end{equation}
and
\begin{equation} \nonumber
\|g^*\|_p
\ \asymp \
\lambda 2^{- \alpha |k^*|_1 - \beta |k^*|_\infty}.
\end{equation}
Hence, by Corollary  \ref{Corollary[B-inequality]} there is $\lambda > 0$ independent of $\eta$ and $n$ such that
$g^* \in \Uab$. Notice that $M_{k^*,s+m-1}(x), \ x \not\in J(s)$, for every $s \in Z^*(\eta)$, and consequently, by \eqref{cap[X_n]} $g^*(x^j) = 0, \ j =1,...,n$. From the inequality \eqref{[r_n(W, L_q)>]} 
\eqref{asymp[|g^*|_q]} and \eqref{def[eta]} we obtain
\begin{equation} \nonumber
r_n 
\ \gg \
\|g^*\|_q 
\ \asymp \
n^{- \alpha  - \beta/d + 1/p - 1/q}.
\end{equation}
This proves the lower bound of \eqref{asymp[r_nUab]} for the case $\beta > 0$.

We now consider the case $\beta < 0$. We will use some notations which coincide with those in the proof of the case $\beta > 0$. Put $k^*= k^*(\eta) := (\eta,m,...,m)$ for integer $\eta > m$. Consider the boxes  $J(s) \subset \IId$
\begin{equation} \nonumber
J(s)
:= \
\{x \in \IId: 2^{- \eta + m} s_1 \le x_1 < 2^{- \eta + m} (s_1 +1)\}, \ s \in Z(\eta), 
\end{equation}
where
\begin{equation} \nonumber
Z(\eta)
:= \
\{s \in \ZZdp: 0 \le  s_1 \le  2^{\eta-m} - 1, \ s_j = 0, \ j = 2,...,d\}. 
\end{equation}
For a given $n$, we find $\eta$ satisfying the relations 
\begin{equation}  \label{def[eta(beta<0)]}
n 
\ \asymp \
2^{k^*_1}
\ \asymp \
2^{\eta - m}
\ = \ 
|Z(\eta)|
\ \ge \ 
2n.
\end{equation}
Let $X_n = \{x^j\}_{j=1}^n$ be an arbitrary subset of $n$ points in $\IId$. Since $J(s) \cap J(s') = \varnothing$ for 
$s \not= s'$, and $|Z(\eta)| \ge 2n$, there is $Z^*(\eta) \subset Z(\eta)$ such that $|Z^*(\eta)| \ge n$ and
\begin{equation} \label{cap[X_n(beta<0)]}
X_n \cap \{\cup_{s \in Z^*(\eta)} J(s)\} \ = \ \varnothing.
\end{equation}
Consider the function $g^* \in \Sigma_r^d(k^*)$ defined by
\begin{equation} \label{[g^*Uab(2)]}
g^*
:= \
\lambda 2^{- (\alpha + \beta - 1/p)k^*_1} \sum_{s \in Z^*(\eta)} M_{k^*,s+r'/2},
\end{equation}
where $M_{k^*,s+r'/2}$ are B-splines of order $r'$. Since $|Z^*(\eta)| \asymp 2^{k^*_1}$, by \eqref{eq:StabIneq} we have
\begin{equation} \label{asymp[|g^*|_q(beta<0)]}
\|g^*\|_q
\ \asymp \
\lambda 2^{- (\alpha + \beta - 1/p + 1/q)k^*_1},
\end{equation}
and
\begin{equation} \nonumber
\|g^*\|_p
\ \asymp \
\lambda 2^{- (\alpha + \beta)k^*_1}.
\end{equation}
Hence, by Corollary  \ref{Corollary[B-inequality]} there is $\lambda > 0$ independent of $\eta$ and $n$ such that
$g^* \in \Uab$. Notice that $M_{k^*,s+m-1}(x), \ x \not\in J(s)$, for every $s \in Z^*(\eta)$, and consequently, by \eqref{cap[X_n(beta<0)]} $g^*(x^j) = 0, \ j =1,...,n$. From the inequality \eqref{[r_n(W, L_q)>]} 
\eqref{asymp[|g^*|_q(beta<0)]} and \eqref{def[eta(beta<0)]} we obtain
\begin{equation} \nonumber
r_n(\Uab, L_q) 
\ \gg \
\|g^*\|_q 
\ \asymp \
n^{- \alpha  - \beta + 1/p - 1/q}.
\end{equation}
This proves the lower bound of \eqref{asymp[r_nUab]} for the case $\beta < 0$.
\end{proof}

\begin{theorem} \nonumber
Let $\ 0 < p, \theta, q \le \infty$ and $a \in \RRdp$ satisfying the condition \eqref{condition[a]} 
 and
\begin{equation} \nonumber
1/p < a_1 < a_2 \le ... \le a_d < r. 
\end{equation}
 Assume that for a given $n \in \ZZ_+$, $\xi_n$ is the largest nonnegative number such that
\begin{equation} \label{[|G(Delta'(xi_n)|<]}
|G(\Delta'(\xi_n))| \ \le \ n.
\end{equation}  
Then $R_{\Delta(\xi_n)}$ defines an asymptotically optimal linear sampling algorithm for
$r_n:= r_n(\Uab, L_q(\IId))$ by
\begin{equation} \nonumber
R_{\Delta'(\xi_n)}(f) 
\ = \ 
L_n(X_n^*,\Phi_n^*,f)
\ = \ 
\sum_{(k,s) \in K(\Delta'(\xi_n))}  f(2^{-k}s) \psi_{k,s}, 
\end{equation} 
where $X_n^* := G(\Delta'(\xi_n))= \{ 2^{-k}s\}_{(k,s) \in K(\Delta'(\xi_n))}$, 
$\Phi_n^*:= \{ \psi_{k,s}\}_{(k,s) \in K(\Delta'(\xi_n))}$,
and we have the following asymptotic order
\begin{equation} \label{asymp[r_nUa]}
\sup_{f \in \Ua}\ \|f - R_{\Delta'(\xi_n)}(f) \|_q
\ \asymp \
r_n 
\ \asymp \
n^{- a_1 + (1/p - 1/q)_+}.
\end{equation}
\end{theorem}

\begin{proof} \\
{\em Upper bounds}.
For a given $n \in \ZZ_+$ (large enough), due to Lemma \ref{Lemma[SamplingNumberUa]} we can define $\xi = \xi_n$ as the largest nonnegative number such that
\begin{equation}  \nonumber
n 
\ \asymp \
2^{\xi_n/(a_1 - (1/p - 1/q)_+)}
\ \asymp \ 
|G(\Delta'(\xi_n))|
\ \le \ 
n.
\end{equation}
Hence, we find
\begin{equation} \label{asymp[2^xiUa]}
2^{-\xi_n}
\ \asymp \
n^{- a_1 + (1/p - 1/q)_+}.
\end{equation}
By Lemma \ref{Lemma:R_n=L_n} and \eqref{[|G(Delta'(xi_n)|<]} $R_{\Delta'(\xi_n)}$ is a linear sampling algorithm of the form 
\eqref{def[L_n]} and consequently, from  Theorem \ref{Ineq[UpperBoundUa]} we get
\begin{equation} \nonumber
r_n 
\ \le \
\sup_{f \in \Ua} \ \|f - R_{\Delta'(\xi_n)}(f)\|_q 
\ \ll \ 
2^{-\xi_n}.
\end{equation}
These relations together with \eqref{asymp[2^xiUa]} proves the upper bounds for \eqref{asymp[r_nUa]}.

\noindent
{\em Lower bounds}. As in the proof of Theorem \ref{Theorem[r_nUab]}, it is sufficient to prove the lower bound for the case $p \le q$.
Fix a number $r = 2^m$ with integer $m$ so that $r_d <  \min(r, r - 1 + 1/p)$.  In the next steps, the proof is similar to the proof of the lower bound for the case $\beta < 0$ in Theorem \ref{Theorem[r_nUab]}. Indeed, we can  repeat almost all the details in it with replacing $\alpha + \beta$ by $a_1$. 
\end{proof}

\medskip
\noindent
{\bf Remark} Concerning the asymptotically optimal sparse grids of sampling points $G(\Delta(\xi_n))$ and $G(\Delta'(\xi_n))$ for $r_n(\Uab, L_q) $ and 
$r_n(\Ua, L_q)$, it is worth to notice the following.  Let set $A$ and $B$ be the sets of  triples $(p,\theta,q)$ introduced in 
Section \ref{Sampling recovery}. For every triple $(p,\theta,q) \in A$, we can define the best choice of family of asymptotically optimal sparse grids $G(\Delta(\xi_n))$ and $G(\Delta'(\xi_n))$. Whereas, for a triple 
$(p,\theta,q) \in B$, there are many families of asymptotically optimal sparse grids $G(\Delta(\xi_n))$ and $G(\Delta'(\xi_n))$ depending on parameter 
$\varepsilon >0$, for $r_n(\Uab, L_q)$ and $r_n(\Ua, L_q)$, respectively. 
 Moreover, the parameter $\varepsilon >0$ plays a crucial role in the construction of asymptotically optimal sparse grids for $(p,\theta,q) \in B$. Indeed, to understand the substance let us consider, for instance, the problem of asymptotically optimal sparse grids for even the simplest case 
$r_n(U^{\alpha,\beta}_{2,2}, L_2))$ with $\beta <0$. Suppose that for this case instead the set
\begin{equation*} 
\Delta (\xi)
:= \
\{k \in \ZZdp: (\alpha  - \varepsilon)|k|_1 + (\beta + \varepsilon) |k|_\infty\le \xi\},
\end{equation*}
we take the set
\begin{equation*} 
\tilde {\Delta} (\xi)
:= \
\{k \in \ZZdp: \alpha|k|_1 + \beta |k|_\infty \le \xi\}.
\end{equation*}
Then  $\tilde {\Delta} (\xi)$ is a proper subset of $\Delta (\xi)$, i.e., the grid $G(\Delta(\xi))$ is essentially extended from $G(\tilde {\Delta} (\xi))$ by parameter $\varepsilon$. However, 
$|G(\Delta (\xi))|\asymp |G(\tilde {\Delta}(\xi))|$. On the other hand, the grid $G(\tilde {\Delta} (\xi))$ cannot be asymptotically optimal for  $r_n(U^{\alpha,\beta}_{2,2}, L_2)$ and 
because  for this gird \eqref{Ineq[UpperBoundUab]} is replaced by
\begin{equation}  \nonumber
\sup_{f \in U^{\alpha,\beta}_{2,2}} \ \|f - R_{\tilde{\Delta}(\xi)}(f)\|_2 
\ \asymp \ 
2^{-\xi} \xi^{d-1}.
\end{equation}
A similar optimal property of the grid $G(\Delta''(\xi))$ holds for  $r_n(U^{\alpha,\beta}_{2,2}, W^\gamma_2(\IId))$ with $\gamma > \beta$ (see the next section).

\section{Sampling recovery in energy norm } \label{Sampling recovery in Bg}

In this section, we extend the results on sampling recovery  in the quasi-norm of $L_q(\IId)$ of  functions from $\Bab$
 in Sections \ref{Sampling recovery} and \ref{Optimality}, to sampling recovery in the energy norm of the isotropic Sobolev space $W^\gamma_q(\IId)$ with $\gamma >0$ and $1<q<\infty$. We preliminarily study the sampling recovery in the norm of $\Bgq$, and then receive results on sampling recovery in the norm of $W^\gamma_q(\IId)$ as consequences of those in the norm of $B^{\gamma}_{q,\min(q,2)}$ and  the inequality \eqref{ineq[normsW,B]}. 
 
Put $\theta':= (1 - 1/\theta)^{-1}$ for $1 \le \theta \le \infty$.

\begin{lemma} \label{Lemma[|f - R_Delta(f)|_Bgq<]}
Let $0 < p, \theta \le \infty$, $1 \le q,\tau \le \infty$, $0 < \gamma < \min(r, r - 1 + 1/p)$ and  
$\psi: \ZZdp \to \RR_+$.  Then for every  $f \in \Bpsi$, we have 
\begin{equation} \label{eq[|f - R(f)|_{Bgq}]}
\|f - R_\Delta(f)\|_{\Bgq} 
\ \ll \ 
\|f\|_{\Bpsi}
\begin{cases}
\displaystyle
\sup_{k \in \ZZdp \setminus \Delta} 2^{-\psi(k) + \gamma |k|_\infty +(1/p - 1/q)_+ |k|_1}, \ & \theta \le 1, \\
\left(\sum_{k \in \ZZdp \setminus \Delta} \{2^{-\psi(k) + \gamma |k|_\infty +(1/p - 1/q)_+ |k|_1}\}^{\theta'}\right)^{1/\theta'}, \ & \theta > 1.
\end{cases}
\end{equation}
\end{lemma}

\begin{proof}
Let us prove the following Bernstein type inequality 
\begin{equation} \label{B-ineq}
 \|g\|_{\Bgq} 
\ \ll \ 2^{\gamma |k|_\infty}\|g\|_q, \ \forall g \in \Sigma^d_r(k),  \ \forall k \in \ZZdp,
\end{equation}
where we recall that $\Sigma^d_r(k)$ is the set of B-splines of the form \eqref{def:StabIneq}.
Let $g \in \Sigma^d_r(k)$ and  $k \in \ZZdp$.
Due to \eqref{NormB^beta(5)} it is sufficient to  show that
\begin{equation} \label{B-ineq(2)}
\biggl(\sum_{s \in \ZZ_+}\left\{2^{\gamma s} \omega_{r,i}(g,2^{-s})_q\right\}^\tau\biggl)^{1/\tau}  
\ \ll \ 
 2^{\gamma k_i}\|g\|_q
\end{equation}
with the usual change to sup when $\tau = \infty$.
Observe that $g$ can be represented by  the  series 
\eqref{representationby[Sigma_{r,i}^d(k)_p]} with
$g_s = g \in \Sigma^d_{r,i}(k_i)_q$ for $s = k_i$ and $g_s = 0$ for $s \not= k_i$.
Hence, applying Lemma \ref{Lemma[quasi-norm<]} we prove  \eqref{B-ineq(2)} and therefore, \eqref{B-ineq}.
The inequality \eqref{B-ineq} together with \eqref{eq:StabIneq} gives
\begin{equation}  \label{B-ineq(3)}
  \|g\|_{\Bgq} 
\ \ll \ 2^{\gamma |k|_\infty +(1/p - 1/q)_+ |k|_1}\|g\|_p, 
\ \forall g \in \Sigma^d_r(k),  \ \forall k \in \ZZdp.
\end{equation}
Since  $1 \le q,\tau \le \infty$, $\|\cdot\|_{\Bgq}$ is a norm.  Consequently, from the inclusions 
$q_k(f) \in \Sigma^d_r(k)$ and
\eqref{B-ineq(3)} we obtain for every $f \in \Bpsi$, 
\begin{equation*} 
\|f - R_\Delta(f)\|_{\Bgq} 
\  \le \ 
 \sum_{k \in \ZZdp \setminus \Delta} \| q_k(f)\|_{\Bgq}
\  \ll \ 
 \sum_{k \in \ZZdp \setminus \Delta} 2^{\gamma |k|_\infty +(1/p - 1/q)_+ |k|_1} \|q_k(f)\|_p.
\end{equation*}
By use of these inequalities, in a way similar to the proof of Lemma \ref{Lemma[|f - R_Delta(f)|_q<]} we prove the lemma.
\end{proof}

Let $0 < p, \theta \le \infty$, $1 \le q,\tau \le \infty$ and $\alpha, \gamma \in \RR_+$, $\beta \in \RR$ be given.
We fix a number $\varepsilon$ so that
\begin{equation*} 
0 < \varepsilon < \min (\alpha - (1/p - 1/q)_+, |\gamma - \beta|),
\end{equation*} 
 and define the set $\Delta^{''}(\xi)$ for $\xi > 0$ by
\begin{equation*} 
\Delta^{''}(\xi)
:= \
\begin{cases}
\{k \in \ZZdp: (\alpha - (1/p - 1/q)_+)|k|_1 - (\gamma - \beta)|k|_\infty \le \xi\}, \ & \theta \le 1, \\
\{k \in \ZZdp: (\alpha - (1/p - 1/q)_+ + \varepsilon/d)|k|_1 - (\gamma - \beta - \varepsilon) |k|_\infty\le \xi\},
 \ & \theta > 1, \ \beta > \gamma, \\
\{k \in \ZZdp: (\alpha - (1/p - 1/q)_+ - \varepsilon)|k|_1 - (\gamma - \beta  + \varepsilon) |k|_\infty\le \xi\},
 \ & \theta > 1, \ \beta < \gamma.
\end{cases}
\end{equation*}

The following theorem gives an upper bound of the error in the quasi-norm of $\Bgq$ of the sampling recovery by the linear sampling operator $R_{\Delta^{''}(\xi)}$ on the grids $G(\Delta^{''}(\xi))$ for 
$f \in \Uab$. It can be proven in a similar way to the proof of Theorem~\ref{Theorem[UpperBoundUab]} with a slight modification on the basis of Lemma~\ref{Lemma[|f - R_Delta(f)|_Bgq<]}. 
 
\begin{theorem} \label{Theorem[UpperBoundUabInBgq]}
Let $0 < p, \theta \le \infty$, $1 \le q,\tau \le \infty$, $\alpha, \gamma \in \RR_+$ and  
$\beta \in \RR$, $\beta \not= \gamma$, satisfy the conditions
\begin{equation} \nonumber
\alpha > 
\begin{cases}
(\gamma -  \beta)/d, \ & \beta > \gamma, \\
\gamma -  \beta, \ & \beta < \gamma,
\end{cases} 
\end{equation}
and
\begin{equation} \nonumber
 1/p < \min (\alpha,\alpha + \beta) \le  \max(\alpha,\alpha + \beta) < r, 
\quad  0 < \gamma < \min(r, r - 1 + 1/p). 
\end{equation}
Then we have the following upper bound
\begin{equation}  \label{ineq[|f - R(f)|_{Bgq}]}
\sup_{f \in \Uab} \ \|f - R_{\Delta^{''}(\xi)}(f)\|_{\Bgq} 
\ \ll \ 
2^{-\xi}.
\end{equation}
\end{theorem}

As the next step we need the asymptotic order of the cardinality of the grids $G(\Delta^{''}(\xi))$. It can be established in a similar way to Lemma~\ref{Lemma[SamplingNumber]} with a slight modification. More precisely, we have the following

\begin{lemma} \label{Lemma[SamplingNumber'']}
 Under the assumptions of Theorem \ref{Theorem[UpperBoundUabInBgq]} we have 
\begin{equation}  \label{asymp[|G(Delta^{''}|]}
|G(\Delta^{''}(\xi))|
\ \asymp \ 
\sum_{k \in \Delta^{''}(\xi)} 2^{|k|_1}
\ \asymp \ 
2^{\xi/ \nu},
\end{equation}
where 
\begin{equation} \nonumber
\nu 
:= \
\begin{cases}
\alpha  - (\gamma - \beta)/d - (1/p - 1/q)_+, \ &  \beta > \gamma, \\
\alpha  - (\gamma - \beta) - (1/p - 1/q)_+, \ &  \beta < \gamma.
\end{cases}
\end{equation}
\end{lemma}

From Theorem~\ref{Theorem[UpperBoundUabInBgq]} and Lemma~\ref{Lemma[SamplingNumber'']}, by analogous technique and argument as in the proof of Theorem~\ref{Theorem[r_nUab]} we prove the following

\begin{theorem} \label{Theorem[r_nUabInBgq]}
Under the assumptions of Theorem \ref{Theorem[UpperBoundUabInBgq]}, let for a given $n \in \ZZ_+$, $\xi_n$ be the largest nonnegative number such that
\begin{equation} \nonumber
|G(\Delta^{''}(\xi_n))| \ \le \ n.
\end{equation}  
Then $R_{\Delta^{''}(\xi_n)}$ defines an asymptotically optimal linear sampling algorithm for
$r_n := r_n(\Uab, \Bgq) $ by
\begin{equation} \nonumber
R_{\Delta^{''}(\xi_n)}(f) 
\ = \ 
L_n(X_n^*,\Phi_n^*,f)
\ = \ 
\sum_{(k,s) \in K(\Delta^{''}(\xi_n))}  f(2^{-k}s) \psi_{k,s}, 
\end{equation} 
where $X_n^* := G(\Delta^{''}(\xi_n))= \{ 2^{-k}s\}_{(k,s) \in K(\Delta^{''}(\xi_n))}$, 
$\Phi_n^*:= \{ \psi_{k,s}\}_{(k,s) \in K(\Delta^{''}(\xi_n))}$,
and we have the following asymptotic orders
\begin{equation} \label{asymp[r_nUabInBgq]}
\sup_{f \in \Uab}\ \|f - R_{\Delta^{''}(\xi_n)}(f) \|_{\Bgq}
\asymp \
r_n 
\ \asymp 
\begin{cases}
n^{- \alpha  - (\beta - \gamma)/d + (1/p - 1/q)_+},  &  \beta > \gamma, \\
n^{- \alpha  - \beta + \gamma + (1/p - 1/q)_+},  &  \beta < \gamma.
\end{cases}
\end{equation}
\end{theorem}

Notice that although in general, for a set $\Delta$ the upper bound  \eqref{eq[|f - R(f)|_{Bgq}]} depends on $\theta$ and $\tau$, 
for the special set $\Delta^{''}(\xi)$ the upper bound \eqref{ineq[|f - R(f)|_{Bgq}]} does not depend on 
$\theta$ and $\tau$, more precisely, these parameters go to a constant in $\ll$. Hence, the upper bound in 
\eqref{asymp[r_nUabInBgq]} which is derived from \eqref{ineq[|f - R(f)|_{Bgq}]} and \eqref{asymp[|G(Delta^{''}|]},
does not depend on $\theta$ and $\tau$ too.

From Theorem \ref{Theorem[r_nUabInBgq]} and the inequalities \eqref{ineq[normsW,B]} we derive following theorem on optimal sampling recovery in the energy norm of $ W^\gamma_q(\IId)$ of the class $\Uab$.

\begin{theorem} \nonumber
Under the assumptions of Theorem \ref{Theorem[UpperBoundUabInBgq]}, we have the following asymptotic orders for 
$1 < q < \infty$,
\begin{equation} \nonumber
r_n(\Uab, W^\gamma_q(\IId)) 
\ \asymp 
\begin{cases}
n^{- \alpha  - (\beta - \gamma)/d + (1/p - 1/q)_+},  &  \beta > \gamma, \\
n^{- \alpha  - \beta + \gamma + (1/p - 1/q)_+},  &  \beta < \gamma.
\end{cases}
\end{equation}
\end{theorem}

\medskip
\noindent
{\bf Remark} Asymptotically optimal linear sampling algorithms  for $r_n(\Uab, W^\gamma_q(\IId))$ 
are  the same as for  $r_n(\Uab, B^\gamma_{q,\min(q,2)})$. For periodic functions, the asymptotic order of 
$r_n(U^{\alpha,\beta}_{2,2}, W^\gamma_2(\TTd))$ with $\beta < \gamma$ recently has been obtained in \cite{BDSU14}.

\section{Optimal cubature} \label{Numerical integration}

Every linear sampling algorithm $L_n(X_n,\Phi_n,\cdot)$ of the form \eqref{def[L_n]} generates  
the cubature formula $I_n(X_n,\Lambda_n,f)$ where 
\begin{equation*}
\Lambda_n = \{\lambda_j \}_{j =1}^n, \quad \lambda_j = \int_{\IId} \varphi_j(x) \ dx.
\end{equation*} 
Hence, it is easy to see that
\begin{equation*} 
 |I(f) - I_n(X_n,\Lambda_n,f)|
 \ \le \
 \|f - L_n(X_n,\Phi_n,f)\|_1, 
\end{equation*}
and consequently, from the definitions we have the following inequality 
\begin{equation} \label{[i_n<r_n]}
i_n(W) 
\ \le \ 
r_n(W)_1. 
\end{equation}

\begin{theorem} \nonumber
Let $\ 0 < p, \theta \le \infty$ and $\alpha \in \RR_+$, $\beta \in \RR$ such that
\begin{equation} \nonumber
1/p < \min (\alpha,\alpha + \beta) \le  \max(\alpha,\alpha + \beta) < r. 
\end{equation}
Assume that for a given $n \in \ZZ_+$, $\xi_n$ is the largest nonnegative number such that
\begin{equation} \nonumber
|G(\Delta(\xi_n))| \ \le \ n.
\end{equation}  
Then $R_{\Delta(\xi_n)}$ defines an asymptotically optimal cubature formula for
$i_n(\Uab)$ by
\begin{equation}  \nonumber
I_n(X_n^*,\Phi_n^*,f)
\ = \ 
\sum_{(k,s) \in K(\Delta(\xi_n))}  \lambda_{k,s}f(2^{-k}s), 
\end{equation}
where
\begin{equation} \nonumber
X_n^* := G(\Delta(\xi_n))= \{ 2^{-k}s\}_{(k,s) \in K(\Delta(\xi_n))}, \quad
\Lambda_n^*:= \{\lambda_{k,s}\}_{(k,s) \in K(\Delta(\xi_n))} \quad 
\lambda_{(k,s} := \ \int_{\IId}\psi_{k,s}(x) dx,
\end{equation} 
and we have the following asymptotic orders
\begin{equation} \label{asymp[i_nUab]}
\sup_{f \in \Uab}\ |I(f) - I_n(X_n^*,\Lambda_n^*,f)|
\ \asymp \
i_n(\Uab) 
\ \asymp \
\begin{cases}
n^{- \alpha  - \beta/d + (1/p - 1)_+}, \ &  \beta > 0, \\
n^{- \alpha  - \beta + (1/p - 1)_+}, \ &  \beta < 0.
\end{cases}
\end{equation}
\end{theorem}

\begin{proof}
The upper bound of \eqref{asymp[i_nUab]} follows from \eqref{[i_n<r_n]} and Theorem \ref{Theorem[r_nUab]}.
To prove the lower bound of \eqref{asymp[i_nUab]} we observe that
\begin{equation} \nonumber
i_n(W) 
\ \ge \
\inf_{X_n=\{x^j\}_{j=1}^n \subset \IId} \ \ \sup_{f \in W: \ f(x^j) = 0, \ j=1,...,n} \ |I(f)|, 
\end{equation}
and for the functions $g^*$ given in \eqref{[g^*Uab(1)]} and \eqref{[g^*Uab(2)]} we have
$I(g^*)  = \|g^*\|_1$.
Hence, we can see that the lower bound is derived from the proof of the lower bound of Theorem \ref{Theorem[r_nUab]}.
\end{proof}

In a similar way, we can prove the following

\begin{theorem} \nonumber
Let $\ 0 < p, \theta \le \infty$ and $a \in \RRdp$ satisfying the condition \eqref{condition[a]} and
$a_1 > 1/p$. 
Assume that for a given $n \in \ZZ_+$, $\xi_n$ is the largest nonnegative number such that
\begin{equation} \nonumber
|G(\Delta'(\xi_n))| \ \le \ n.
\end{equation}  
Then $R_{\Delta'(\xi_n)}$ defines an asymptotically optimal cubature formula for
$i_n(\Ua)$ by
\begin{equation}  \nonumber
I_n(X_n^*,\Phi_n^*,f)
\ = \ 
\sum_{(k,s) \in K(\Delta'(\xi_n))}  \lambda_{k,s}f(2^{-k}s), 
\end{equation}
where
\begin{equation} \nonumber
X_n^* := G(\Delta'(\xi_n))= \{ 2^{-k}s\}_{(k,s) \in K(\Delta'(\xi_n))}, \quad
\Lambda_n^*:= \{\lambda_{k,s}\}_{(k,s) \in K(\Delta'(\xi_n))} \quad 
\lambda_{k,s} := \ \int_{\IId}\psi_{k,s}(x) dx,
\end{equation} 
and we have the following asymptotic order
\begin{equation} \nonumber
\sup_{f \in \Ua}\ |I(f) - I_n(\Lambda_n^*,X_n^*,f)|
\ \asymp \
i_n(\Ua) 
\ \asymp \
n^{- a_1 + (1/p - 1)_+}.
\end{equation}
\end{theorem}

\noindent
{\bf Acknowledgments.}  This work is funded by Vietnam National Foundation for Science and Technology Development (NAFOSTED) under  Grant No. 102.01-2014.02. A part of this work was done when the author was working as a research professor at the Vietnam Institute for Advanced Study in Mathematics (VIASM). He  would like to thank  the VIASM  for providing a fruitful research environment and working condition. The author would like to specially thank Dr. Tino Ullrich and Glenn Byrenheid for their valuable remarks and suggestions. He thanks the referees for constructive remarks, comments and suggestions which certainly improved the presentation of the paper.

\end{document}